\documentclass[11pt]{article}
\usepackage{amsmath}[1996/11/01]
\usepackage{amssymb,amsthm,amsxtra}
\usepackage{amscd}
\usepackage[margin=1in]{geometry}
\usepackage[title,titletoc]{appendix}

\def\[#1\]{\begin{equation}#1\end{equation}}
\makeatletter
\def\beq{%
   \relax\ifmmode
      \@badmath
   \else
      \ifvmode
         \nointerlineskip
         \makebox[.6\linewidth]%
      \fi
      $$
   \fi
}
\def\eeq{%
   \relax\ifmmode
      \ifinner
         \@badmath
      \else
         $$
      \fi
   \else
      \@badmath
   \fi
   \ignorespaces
}

\def\enddisplaymath{\eeq\global\@ignoretrue}
\makeatother

\newtheorem{thm}{Theorem}
\newtheorem{cor}[thm]{Corollary}
\newtheorem{lem}[thm]{Lemma}
\newtheorem{prop}[thm]{Proposition}

\theoremstyle{remark}
\newtheorem*{rem}{Remark}
\newtheorem{rems}{Remark}[thm]

\theoremstyle{definition}
\newtheorem{defn}{Definition}

\numberwithin{equation}{section}
\numberwithin{thm}{section}
\numberwithin{eg}{section}
\numberwithin{defn}{section}

\DeclareMathOperator{\Mat}{Mat}

\DeclareMathOperator{\GL}{GL}
\DeclareMathOperator{\SL}{SL}
\DeclareMathOperator{\PGL}{PGL}
\DeclareMathOperator{\PSL}{PSL}
\DeclareMathOperator{\GU}{GU}
\DeclareMathOperator{\SU}{SU}
\DeclareMathOperator{\PGU}{PGU}
\DeclareMathOperator{\PSU}{PSU}
\DeclareMathOperator{\GO}{GO}
\DeclareMathOperator{\SO}{SO}
\DeclareMathOperator{\Sp}{Sp}
\DeclareMathOperator{\PSp}{PSp}

\DeclareMathOperator{\Aut}{Aut}

\DeclareMathOperator{\Out}{Out}
\DeclareMathOperator{\Gal}{Gal}

\DeclareMathOperator{\ST}{ST}
\DeclareMathOperator{\Alt}{Alt}
\def\normal{\mathrel{\unlhd}}

\DeclareMathOperator{\Tr}{Tr}
\DeclareMathOperator{\Pic}{Pic}
\DeclareMathOperator{\Prym}{Prym}
\DeclareMathOperator{\Spec}{Spec}
\DeclareMathOperator{\Tor}{Tor}
\DeclareMathOperator{\coker}{coker}
\DeclareMathOperator{\im}{im}

\newcommand{\frakm}{{\mathfrak{m}}}
\newcommand{\frakp}{{\mathfrak{p}}}
\newcommand{\frakg}{{\mathfrak{g}}}

\newcommand{\Q}{\mathbb Q}
\newcommand{\F}{\mathbb F}
\newcommand{\Z}{\mathbb Z}
\renewcommand{\H}{\mathbb H}
\renewcommand{\P}{\mathbb P}
\newcommand{\A}{\mathbb A}
\newcommand{\G}{\mathbb G}
\newcommand{\C}{\mathbb C}
\newcommand{\R}{\mathbb R}

\newcommand{\gl}{{\mathfrak{gl}}}
\renewcommand{\sl}{{\mathfrak{sl}}}
\newcommand{\su}{{\mathfrak{su}}}

\renewcommand{\sp}{{\mathfrak{sp}}}
\DeclareMathOperator{\Lie}{Lie}
\DeclareMathOperator{\Ad}{Ad}

\DeclareMathOperator{\cosoc}{cosoc}
\DeclareMathOperator{\rad}{rad}

\DeclareMathOperator{\ord}{ord}
\DeclareMathOperator{\rnk}{rnk}

\DeclareMathOperator{\Sel}{Sel}
\DeclareMathOperator{\Res}{Res}
\newcommand{\et}{\text{\'et}}
\newcommand{\fl}{\text{fl}}
\newcommand{\sep}{\text{sep}}

\DeclareMathOperator{\U}{U}

\begin{document}

\title{The monodromy of cyclic Pryms}
  \author{Eric M. Rains}

\date{December 22, 2025}
\maketitle

\begin{abstract}
The {\em Prym} of a cyclic covering of smooth projective curves is the
``new'' part of the Jacobian: the quotient of the Jacobian of the covering
curve by the Jacobians of the intermediate covers.  Given a family of such
coverings, the fundamental group of the base of the family acts on the Tate
modules of the Pryms, and the image of this representation is a key
ingredient in answering arithmetic statistics questions about the
distribution of the group structure of the $L$-torsion of a random Prym in
the family.  (Over $\F_q$, the action of Frobenius is roughly uniformly
distributed over the {\em arithmetic} monodromy, a coset of the image of
the fundamental group of the base change to $\bar\F_q$ (the {\em geometric}
monodromy).)  In the present note, we show for a number of natural families
that (with limited exceptions) the geometric monodromy is sandwiched
between a certain unitary group and its derived subgroup.  In particular,
this holds for the one-parameter families obtained by starting with any
fixed cover and varying one (tame) ramification point.  As an application,
we deduce analogous largeness results for the monodromy of the Selmer
groups of elliptic surfaces with $j=0$ or $j=1728$, by relating them to
cyclic covers of degree 6 or 4 respectively, implying that their Selmer
groups do not satisfy the standard heuristics.  For instance, for eliptic
surfaces with $j=0$ of sufficiently large height over $\P^1_{\F_q}$, the
average size of the $l$-Selmer group is $l+3+o_q(1)$ when $l$ (fixed) and
$q$ (large) are both 1 mod 3, compared to $l+1+o_q(1)$ for general elliptic
surfaces.

\end{abstract}

\tableofcontents

\section{Introduction}

The present work arose out of the work of \cite{selmer0}, showing that
Selmer groups of rational elliptic surfaces satisfy the heuristic of
\cite{BKLPR} in a suitable limit.  This led naturally to the question of
whether one could obtain similar results for families for which the
heuristic does not apply.  The calculations of \cite{selmer0} were in two
parts: an arithmetic calculation of the (geometric) monodromy, and a
combinatorial calculation of the limiting distribution implied by the
monodromy and equidistribution.  Thus the first question is: are there
interesting families not satisfying the general heuristic for which one can
compute the monodromy?

The most natural such families are the families with constant $j$-invariant
$0$ or $1728$, where the curves have nontrivial (geometric) automorphism
group, giving additional structure to the Selmer group and forcing the
monodromy group to be smaller.  In particular, it is not too difficult to
guess a close approximation to the monodromy: the automorphism forces it to
be contained in a generalized unitary group, and in the absence of any
additional structure, one expects it to contain the corresponding special
unitary group.  (This expectation is what we call ``large monodromy''
below.)  This would be enough to suggest a revised heuristic, but one would
of course like to know that the monodromy truly is as expected, giving the
original objective of the present note.

The key further insight that leads to our actual main topic is the
observation that (away from characteristics 2 and 3, at least) curves with
$j=0$ are cyclic twists of constant curves: there is a degree 6 cyclic
extension of the field of definition over which it is isomorphic to ones
favorite curve $E_0$ of $j$-invariant over 0 the prime field.  In
particular, the Galois cohomology of $E[l]$ is closely related to the
Galois cohomology of $E_0[l]$, and the relevant restriction and transfer
maps respect the additional conditions cutting out the Selmer groups.  One
thus finds that the Selmer group of $E[l]$ is essentially just the
invariant part of the Selmer group of $E_0[l]$ over the cyclic extension.
Since the latter elliptic surface has only smooth fibers, its Selmer group
is computable using \'etale cohomology, and one finds that it is nothing
other than the tensor product of the $l$-torsion of $\Pic(E_0)$ and the
$l$-torsion of $\Pic(\tilde{C})$ where $\tilde{C}$ is the covering curve.
As a module over the cyclic group algebra $\Z[z]/(z^6-1)$, $\Pic(E_0)$ is
geometrically isomorphic to $\Z[\zeta_6]$, and thus one essentially has
\[
\Sel_l(E)\cong
\Pic(\tilde{C})[l]\otimes_{(\Z/l\Z)[z]/(z^6-1)} (\Z/l\Z)[\zeta_6].
\]
(This remains true as a representation of the monodromy group, with the one
caveat that the monodromy typically acts as a nontrivial scalar on
$(\Z/l\Z)[\zeta_6]$, so the representation will be twisted accordingly.)
In other words, computing the Selmer monodromy of the general $j=0$ curve
reduces up to a twist to computing the monodromy of the $l$-torsion of the
general $6$-cyclic Prym:
\[
\Pic(\tilde{C})\otimes_{\Z[z]/(z^6-1)} \Z[\zeta_6],
\]
and similarly for $j=1728$.

This then leads naturally to the question: given a family of curves
$\tilde{C}$ with an action of the cyclic group of order $N$, when does the
Prym
\[
\Pic(\tilde{C})\otimes_{\Z[z]/(z^N-1)} \Z[\zeta_N]
\]
have large $l$-torsion monodromy?  (Suppose for the moment that $l$, $N$,
and the characteristic are pairwise relatively prime.)  We show that this
holds in significant generality, with the main result being that even if
one fixes the quotient curve and all but one of the ramification points,
the resulting one-parameter family of curves already has large monodromy,
so long as the rank of the Prym as a $\Z[\zeta_N]$-module is at least $5$.
(We also explain precisely what can happen for rank $3$ and $4$.)  In
slightly weaker form, this continues to hold in a number of cases with
$\gcd(l,N)\ne 1$ or even in characteristic dividing $N$.

The argument proceeds in several steps.  The first is to note that the
monodromy of the Prym in characteristic 0 is determined by the monodromy of
appropriate period integrals.  When the base curve is $\P^1$ (and working
geometrically so that we may replace the cyclic group by $\mu_N$), these
integrals satisfy a well-known differential equation (the Jordan-Pochhammer
equation) in the varying ramification point, the monodromy of which is
well-understood.  In particular, one finds that up to a choice of basis,
one can explicitly right down the matrices by which the generators of the
fundamental group act (though it turns out we do not need to do so).  One
is thus left with the question of determining when the group generated by
those explicit matrices is large mod $l$.  We in fact use a
basis-independent description: the generators are pseudo-reflections (they
fix a codimension 1 subspace pointwise) of known determinant and multiply
to a known scalar.  This (under mild conditions that are necessary for
largeness) is enough to characterize the representation up to conjugacy.

To show largeness of this representation, we first consider largeness
modulo each individual prime of $\Q(\zeta_N+\zeta_N^{-1})$.  For this, we
use a well-known classification of irreducible subgroups of
$\GL_n(\bar\F_p)$ generated by pseudo-reflections to show that (with
limited exceptions for small $n$) the residue group is as expected.  (This
classification, though well-known, does not appear to have been collected
in any single place, so we do so below for all $n$.)  One complication that
arises when $\zeta_N+\zeta_N^{-1}\notin \Q$ is that different primes over
the same rational prime give rise to isomorphic residue groups, and thus
largeness modulo the individual primes need not imply largeness over the
product.  On can show, however, that it suffices to show largeness modulo
products of two primes, which essentially reduces to showing that the two
representations are not related by an isomorphism of the residue group.
The finishing step of largeness is then to lift to powers of primes.  For
unramified primes, this uses the general lifting result of
\cite{VasiuA:2003}, except that the statement given there has some errors;
we correct the statement and discuss a corrected proof below.  For ramified
primes, we use more ad hoc methods to construct enough elements of the
kernel of reduction.

The above suffices to prove largeness of the monodromy for cyclic covers of
$\P^1$ in characteristic 0.  There remain two more issues: extending to
higher genus base curves, and extending to finite characteristic (including
characteristic dividing $N$).  For extending to higher genus, we use the
fact that there are covers of {\em singular} curves such that the Prym is
not only smooth, but isomorphic to the Prym of the normalization.  In
particular, this lets us construct families of covers of {\em nodal} curves
of genus $g$ that have the same Pryms as covers of $\P^1$ and thus have
large monodromy.  This implies largeness for smooth base curves; this
strictly speaking only applies to cases in which the curve and ramification
data all vary, but in characteristic 0, the monodromy of subfamilies is
often normal, making largeness relatively straightforward.  For restriction
to finite characteristic, we avoid some of the technical issues by working
with 1-parameter families, letting us reduce them to checking whether a
certain cyclic cover has good reduction.

Putting it all together, we find that for 1-parameter families of cyclic
covers with a moving {\em tame} ramification point, the monodromy of the
Prym is large (away from $p$-torsion and primes where the monodromy is
reducible) as long as the rank of the Prym as a $\Z[\zeta_N]$-module
(determined from the genus of the base curve and the number of ramification
points) is at least 5; similar results apply for families of general curves
as long as the base genus is positive and $N$ is not a power of $p$.

Given the above, our original application to Selmer groups reduces (modulo
verifying that transfer maps respect Selmer groups) to the cases $N=6$ and
$N=4$ for Prym monodromy, and thus the corresponding 1-parameter families
of elliptic surfaces have large Selmer monodromy.  This is enough to
establish that those Selmer groups do not satisfy the heuristic of
\cite{BKLPR}; in particular, for elliptic surfaces with $j=0$ over finite
fields with a cube root of unity, the expected size of the $l$-Selmer group
for $l=1(3)$ is $l+3$, compared to the $l+1$ one has for general curves.
(See Theorem \ref{thm:avgsel_j0} for a more precise statement.)

The plan of the paper is as follows.  Section 2 discusses the monodromy
group of the Jordan-Pochhammer equation, including analogues over general
fields.  Section 3 discusses the general strategy for proving largeness,
and section 4 carries it out modulo individual primes (including a
discussion at the beginning of the full classification of irreducible
subgroups of $\GL_n(\bar\F_p)$ generated by pseudo-reflections), with
section 5 extending that to any square-free ideal.  Section 6 discusses
lifting; subsection 6.1 discusses a graded Lie algebra associated to the
natural filtration of a $p$-adic group, subsection 6.2 gives the corrected
version of \cite{VasiuA:2003}, with subsection 6.3 briefly applying that to
Jordan-Pochhammer representations.  Subsection 6.4 and 6.5 then discuss
lifting from ramified primes, with the latter dealing with the case that
the group modulo the prime is symplectic.

Section 7 discusses Pryms of singular curves, with particular attention to
when the Prym of a degeneration has smooth N\'eron model.  Section 8
discusses local ramification data, with attention to smoothability of
ramified nodes in the tame case, and to lifting to characteristic 0 in the
wild case.  Section 9 uses those results to extend to higher genus base
curves in characteristic 0, and Section 10 pushes everything down to finite
characteristic.  Finally, Section 11 contains the application to Selmer
groups of constant $j$ elliptic surfaces.

{\bf Acknowledgements}.  This paper arose out of discussions with T. Feng
and A. Landesman on next steps to take after \cite{selmer0}, both of whom
patiently listened to the author explain various earlier approaches to the
$j=0$ case.  They also helped figure out and prove the correct formulation
of Lemma \ref{lem:ramdeg} below, as well as giving additional suggestions
to improve exposition.

\section{Jordan-Pochhammer monodromy}

The typical $\mu_N$-cover $C/\P^1$ with marked ramification points takes
the form
\[
y^N = \prod_{0\le i\le n+1} (x-x_i)^{m_i}
\]
with $\sum_i m_i=0(N)$ and $\gcd(m_0,\dots,m_{n+1},N)=1$ lest the cover be
geometrically reducible.  (Note that we have implicitly marked an
unramified point of the covering curve lying over $\infty\in \P^1$.)  The
action of $\mu_N$ on $C$ induces an action on the space of differentials
and thus a $\Z/N\Z$-grading on the space of differentials, with the
holomorphic differentials of weight $d\bmod N$ ($0\le d<N$) having the form
\[
p(x) y^{d-N} dx
\]
where p(x) is a polynomial of degree at most $(1-d/N)(\sum_i m_i)-2$ such
that
\[
\ord_{x_i} p(x) - m_i(1-d/N)>-1
\]
for $0\le i\le n+1$.

For $x\in \R$, let $\{x\}:=x-\lfloor x\rfloor$.

\begin{prop}
  The dimension of the space of differentials of weight $d\ne 0(N)$ is
  \[
  -1+\sum_{0\le i\le n+1} \{-dm_i/N\}
  \]
\end{prop}  

\begin{proof}
  Indeed, if we round up the conditions on the order of vanishing of $p$ at
  $x_i$, we see that it must vanish to order at least $\lfloor
  m_i(1-d/N)\rfloor$ at $x_i$, and has a pole of order at most $-2+\sum_i
  m_i(1-d/N)$ at $\infty$.  In other words, $p$ ranges over the global
  sections of a line bundle of degree
  \[
  -2+\sum_{0\le i\le n+1} \{-dm_i/N\},
  \]
  so has the stated number of global sections unless the degree is $-2$.
  This only happens when $dm_i/N$ is an integer for all $i$, which forces
  $d=0(N)$.
\end{proof}

\begin{cor}
  The dimensions of the spaces of differentials of weight $d$ and $-d$ add
  to $2$ less than the number of points such that $m_id/N$ is not integer.
\end{cor}

\begin{proof}
  The sum is
  \[
  -2+\sum_{0\le i\le n+1} \{-dm_i/N\}+\{dm_i/N\}
  \]
  and $\{x\}+\{-x\}=1$ unless $x\in \Z$.
\end{proof}

Over $\C$, there is a natural ``period integral'' pairing
\[
H_1(C;\Z)\otimes \Gamma(C;\Omega_C)\to \C^*
\]
given by integrating the holomorphic differential along the homology class,
such that the induced map
\[
H_1(C;\Z)\to \Gamma(C;\Omega_C)^*
\]
is injective.  If we let $\Gamma^P(C;\Omega_C)$ denote the subspace
on which $\Z/N\Z$ acts with weight prime to $N$, then we may consider the
composition
\[
H_1(C;\Z)\to \Gamma^P(C;\Omega_C)^*
\]
Since the period integral map is natural, so in particular
$\mu_N$-equivariant, the same applies to the composition.  In particular,
if we fix a generator $g\in \mu_N$, then $\Phi_N(g)$ annihilates the
right-hand side, where $\Phi_N$ is the cyclotomic polynomial, and thus the
map factors through $H_1(C;\Z)\otimes \Z[\zeta_N]$.

\begin{prop}
  The map
  \[
  H_1(C;\Z)\otimes \Z[\zeta_N]\to \Gamma^P(C;\Omega_C)^*
  \]
  is injective.
\end{prop}

\begin{proof}
  A nonzero element of the kernel would have nontrivial integral with {\em
    some} holomorphic differential, which would necessarily lie in the sum
  of subspaces of degree {\em not} prime to $N$.  But $\Phi_N(g)$
  annihilates $H_1(C;\Z)\otimes \Z[\zeta_N]$ and acts invertibly on those
  differentials, giving a contradiction.
\end{proof}

In particular, we see that we can in principle recover the monodromy of the
Prym from the monodromy of the corresponding period integrals.  One way to
obtain a basis of the homology of $C$ is to choose otherwise disjoint
simple closed contours in $\P^1$ from $\infty$ around each marked point in
$\P^1$ (i.e., standard generators of the fundamental group).  Since we have
chosen a preimage of $\infty$, each such contour has a marked preimage in
$C$, which in turn generates a $\mu_N$-orbit of preimages.  If $\gamma_i$
is the marked preimage in $C$ of the loop around $x_i$,
\[
\partial \gamma_i = g^{m_i}[\infty]-[\infty],
\]
There are also boundaries coming from the regions into which the original
contours separate $\P^1$, generated by
\[
(\sum_{0\le j<N/\gcd(N,m_i)} g^{m_ij}) \gamma_i
\]
for $0\le i\le n+1$ and
\[
\gamma_0+g^{m_0}\gamma_1+g^{m_0+m_1}\gamma_2+\cdots.
\]
It is not too difficult to see that $H^1(C;\Z)\otimes \Z[\zeta_N]$ is
the middle homology of the tensor product complex, but this is somewhat
inconvenient to describe uniformly.  However, over $\Q[\zeta_N]$,
there are a large class of such descriptions coming from the fact that
\[
(g^{m_i}-1)\gamma_j-(g^{m_j}-1)\gamma_i
\]
for $i\ne j$ are all cocycles, and any collection of $n$ such cocycles for
which the corresponding graph has no cycles gives a basis of
$H^1(C;\Z)\otimes \Q[\zeta_N]$.  Since we are now working rationally,
it will be convenient to consider the rescaled classes
\[
\gamma_{ij}:=(\zeta_N^{m_j}-1)^{-1}\gamma_j-(\zeta_N^{m_i}-1)^{-1}\gamma_i.
\]
The significance of this is that the integral of a holomorphic differential
along $\gamma_{ij}$ is equal to the (improper) integral of the differential
along a contour from $x_i$ to $x_j$.  Indeed, the integral along $\gamma_i$
is $\zeta_N^{m_i}-1$ times the integral along a contour $\gamma_i'$ from
$\infty$ to $x_i$, since $\gamma_i$ is homotopic in $C$ to a contour that
first proceeds from the marked preimage of $\infty$ along the preimage of
$\gamma_i'$ and then returns to $g^{m_i}\infty$ along {\em that} preimage
of $\gamma'_i$.

In particular, if we can compute the action of monodromy on the integrals
$\int_{x_i}^{x_j} \omega$, this will tell us the rational monodromy of the
Prym.  (In fact, it is then not too difficult to recover the integral
monodromy, but the rational monodromy will suffice for us, as in the cases
of interest there is a unique homothety class of lattices.)  These are
special cases of hypergeometric integrals
\[
\int_{x_i}^{x_{i+1}}
\prod_{0\le i\le n+1} (x-x_i)^{\alpha_i-1}
dx,
\]
and in particular satisfy nice differential equations in $x_0$ (essentially
the Jordan-Pochhammer equation, up to an overall normalization factor).
The following can thus be read off from the local structure of the
Jordan-Pochhammer equation.  Call an element $g\in\GL_n$ a {\em
  pseudoreflection} if $\rnk(g-1)=1$, and refer to its other eigenvalue
as the ``nontrivial'' eigenvalue (which may equal $1$).

\begin{prop} 
  If none of the quantities $\alpha_i$ or $\alpha_0+\alpha_i$ are integers,
  then the monodromy in $x_0$ of the hypergeometric integrals
  \[
  \int_{x_i}^{x_{i+1}} \prod_{0\le i\le n+1} (x-x_i)^{\alpha_i-1} dx
  \]
  for $1\le i\le n$ is generated by elements $g_i$, $1\le i\le n+1$
  such that each $g_i$ is a pseudoreflection with nontrivial eigenvalue
  \[
  \exp(2\pi\sqrt{-1}(\alpha_0+\alpha_i)),
  \]
  and $\prod_i g_i = \exp(2\pi\sqrt{-1}\alpha_0)$.
\end{prop}

We can say a bit more, in fact: if we first change to a basis that contains
$\gamma_{0i}$ and no other edge involving $x_0$, then the integrals
associated to those other edges are all fixed by the monodromy around
$x_i$, while the integral along $\gamma_{0i}$ has eigenvalue
$\exp(2\pi\sqrt{-1}(\alpha_0+\alpha_1))$.  Indeed, if we reparametrize
$x_0=x_i+uy$, then the monodromy of $x_0$ around $x_i$ is the monodromy of
a small loop in $u$ around $0$.  The integrals along edges not containing
$0$ are holomorphic at $u=0$, so have trivial monodromy, while the integral
along $\gamma_{0i}$ is $u^{\alpha_0+\alpha_i}$ times a holomorphic
function.  Since the changes of basis can be made explicit, this lets one
write the monodromy in any desired basis.  (This only applies directly to
cases in which the generators are diagonalizable and none of the exponents
are integers, but the monodromy is clearly meromorphic in the parameters,
so the resulting expression extends.)

In particular, we see that when the $\alpha_i$ are rationals, then the
monodromy is defined over the corresponding cyclotomic field.  Moreover, if
$\beta_i$ is another set of parameters with the same common denominator
such that $\beta_i-d\alpha_i\in \Z$ for some fixed $d$, then {\em its}
monodromy is related by Galois.  We thus find that the different choices of
holomorphic differentials indeed give compatible expressions for the
monodromy of the Prym.

\begin{prop}
  Let $C/\P^1$ be a $\mu_N$-cover of $\P^1_{\C}$ ramified at
  $x_0,\dots,x_{n+1}\in \A^1_{\C}$ with weights $m_0,\dots,m_{n+1}$, and
  consider the family of $\mu_N$-covers as $x_0$ varies (fixing a point
  over $\infty$).  Then the monodromy of the Prym of this family is
  generated by elements $g_1,\dots,g_{n+1}$ such that each $g_i$
  is a pseudoreflection with nontrivial eigenvalue $\zeta_N^{m_0+m_i}$,
  and $\prod_i g_i = \zeta_N^{m_0}$.
\end{prop}

\begin{rem}
Note that although we know the monodromy over $\Z(\zeta_N)$ is generated by
pseudoreflections, it is a priori possible for those pseudoreflections to
become $1$ mod some prime.
\end{rem}

More generally, given elements $\lambda_0,\dots,\lambda_{n+1}\in K^\times$
for some field $K$ (possibly of finite characteristic) with product $1$,
the {\em Jordan-Pochhammer group} is (up to conjugacy) the subgroup of
$\GL_n(K)$ generated by elements $g_1,\dots,g_{n+1}$ such that
$\rnk(g_i-1)\le 1$, $\det(g_i)=\lambda_0\lambda_i$, and
\[
\prod_i g_i = \lambda_0.
\]
(This forces $\prod_{0\le i\le n+1}\lambda_i=1$.)  Of course, the word
``the'' here needs to be justified.  It is, in fact, shown in
\cite{VolkleinH:1998} that there is a unique (up to conjugacy) such tuple
of matrices such that the elements are pseudoreflections and corresponding
representation is irreducible.  Thus the question reduces to showing the
pseudoreflection condition (i.e., that no $g_i$ is 1) and irreducibility.

\begin{lem}
  Let $V$ be a vector space of dimension $\ge n$, and let $g_1,\dots,g_n\in
  \GL(V)$ be such that $\rnk(g_i-1)\le 1$ and $\prod_ig_i$ is a scalar
  $\mu$.  Then $\mu=1$ or $\det(g_i)\equiv \mu$.
\end{lem}

\begin{proof}
  We first note that if $\rnk(g-1)\le a$, $\rnk(h-1)\le b$, then
  $\rnk(gh-1)\le a+b$.  It follows, in particular, that
  \[
  (\prod_{1\le i\le n-1}g_i)^{-1}-1
  \]
  has rank at most $n-1<\dim(V)$.  But this is equal to $\mu^{-1}g_n-1$ and
  thus $\mu$ must be an eigenvalue of $g_n$.  In particular, either $\mu=1$
  (forced if $n<\dim(V)$, since then the eigenspace of eigenvalue $\mu$ is
  at least $2$-dimensional) or $\mu$ is the nontrivial eigenvalue of $g_n$.
  But then cyclic symmetry implies the same for the other $g_i$ as well.
\end{proof}

\begin{lem}
  Let $V$ be a vector space of dimension $n$, and let $g_1,\dots,g_{n+1}\in
  \GL(V)$ be such that $\rnk(g_i-1)\le 1$ and $\prod_i g_i$ is the scalar
  $\mu$.  If $\mu\ne 1$ and is not an eigenvalue of any $g_i$, then
  the $g_i$ are pseudoreflections and generate an irreducible subgroup of
  $\GL(V)$.
\end{lem}

\begin{proof}
  That the $g_i$ are pseudoreflections follows from the previous Lemma, as
  otherwise we effectively have $\le n$ generators.  So it remains only to
  prove irreducibility.  Let $W$ be an invariant subspace.  This determines
  two smaller representations of the free group, $W$ and $V/W$, with each
  $g_i$ either acting trivially on both or as the identity on one and a
  pseudoreflection on the other.  If at most $\dim(W)$ of the $g_i$ act
  nontrivially on $W$ or at most $\dim(V/W)$ act nontrivially on $V/W$,
  then the previous Lemma tells us (since $\mu\ne 1$) that those
  pseudoreflections have nontrivial eigenvalue $\mu$, contradicting our
  hypothesis.  It follows that at least $\dim(W)+1$ elements act
  nontrivially on $W$ and at least $\dim(V/W)+1$ elements act nontrivially
  on $V/W$.  These sets are disjoint and together account for at least
  $(\dim(W)+1)+(\dim(V/W)+1)=n+2>n+1$ generators, giving a contradiction.
\end{proof}

In other words, the Jordan-Pochhammer representation is well-defined and
irreducible as long none of the parameters is $1$.

\begin{lem}
  For each $1\le m\le n$, the image of the product
  $g_{\{1,\dots,m\}}:=g_1\cdots g_m$ under the Jordan-Pochhammer
  representation with parameters
  $(\lambda_0;\lambda_1,\dots,\lambda_{n+1})$ satisfies
  \[
  \dim\ker(g_{\{1,\dots,m\}}-1)=n-m,\qquad
  \dim\ker(g_{\{1,\dots,m\}}-\lambda_0)=m-1
  \]
  with remaining eigenvalue equal to $\lambda_0\cdots\lambda_m$.
\end{lem}

\begin{proof}
  Since $g_{\{1,\dots,m\}}$ is a product of $m$ pseudoreflections,
  \[
  g_{\{1,\dots,m\}}-1
  \]
  has rank at most $m$, while
  \[
  g_{m+1}\cdots g_n-1
  \]
  has rank at most $n-m$, and
  \[
  g_{\{1,\dots,m\}}g_{m+1}\cdots g_n-1 = \lambda_0 g_{n+1}^{-1}-1
  \]
  has rank $n$.  Since the ranks at most add, the various inequalities must
  be tight, and thus
  \[
  \dim\ker(g_{\{1,\dots,m\}}-1)=n-m
  \]
  as required.  By the same argument,
  \[
  \dim\ker(g_{\{1,\dots,m\}}-\lambda_0)=m-1,
  \]
  and the remaining eigenvalue is
  \[
  \det(g_{\{1,\dots,m\}}) \lambda_0^{1-m}
  =
  \lambda_0^{1-m}\prod_{1\le i\le m}\det(g_i)
  =
  \lambda_0^{1-m}\prod_{1\le i\le m}\lambda_0\lambda_i.
  =
  \prod_{0\le i\le m}\lambda_i
  \]
  as required.
\end{proof}

\begin{lem}
  If $\lambda_0\cdots\lambda_m\ne 1$, then the restriction of
  $g_1,\dots,g_m,\lambda_0g_{\{1,\dots,m\}}^{-1}$ to
  $\im(g_{\{1,\dots,m\}}-1)$ is the Jordan-Pochhammer representation with
  parameters
  $(\lambda_0;\lambda_1,\dots,\lambda_m,1/\lambda_0\cdots\lambda_m)$.
\end{lem}

\begin{proof}
  The space $\im(g_{\{1,\dots,m\}}-1)$ is invariant under $g_1,\dots,g_m$,
  and thus their restrictions are either pseudoreflections or trivial,
  as is $\lambda_0 g_{\{1,\dots,m\}}^{-1}$, and thus gives a
  Jordan-Pochhammer representation with the given parameters.  Since none
  of the parameters are 1, it is {\em the} Jordan-Pochhammer representation
  with those parameters.
\end{proof}

This can be significantly generalized.  There is a well-known action of the
braid group on the free group in $n+1$ generators in which $\sigma_i$ acts
by
\[
(g_1,\dots,g_{i-1},g_{i+1},g_{i+1}^{-1}g_ig_{i+1},g_{i+2},\dots,g_{n+1}).
\]
The composition of a Jordan-Pochhammer representation with $\sigma_i$ is
still a representation by pseudo-reflections multiplying to (the same)
scalar, and thus is again a Jordan-Pochhammer representation, with
$\lambda_i$ and $\lambda_{i+1}$ swapped.  For any subset $S\subset
\{1,\dots,n+1\}$, we may apply a sequence of such transformations to get a
representation of the same form in which the tuple starts with $g_i$ for
$i\in S$ (in order).  With this in mind, define $g_S$ to be the {\em
  ordered} product of $g_i$ for $i\in S$, and let $\lambda_S:=\prod_{i\in
  S}\lambda_i$.

\begin{lem}
  The image of $g_S$ under the Jordan-Pochhammer
  representation with parameters
  $(\lambda_0;\lambda_1,\dots,\lambda_{n+1})$
  satisfies
  \[
  \dim\ker(g_S-1)=n-|S|,\qquad
  \dim\ker(g_S-\lambda_0)=|S|-1
  \]
  with remaining eigenvalue $\lambda_0\lambda_S$.  If $\lambda_S\ne 1$,
  then the restriction of $(g_i:i\in S,\lambda_0g_S^{-1})$ to
  $\im(g_S-1)$ is the Jordan-Pochhammer representation with parameters
  $(\lambda_0;\lambda_i:i\in S,1/\lambda_0\lambda_S)$.
\end{lem}

Note that by rigidity and irreducibility of the Jordan-Pochhammer
representation, any element of the braid group that acts trivially on the
parameters (e.g., elements of the {\em pure} braid group) gives rise to a
unique (up to scalars) matrix normalizing the Jordan-Pochhammer
representation.  The generators of the pure braid group correspond to loops
in which some {\em other} ramification point moves, and thus are themselves
given by pseudo-reflections.  Combining this with the fundamental group
gives a representation (the Gassner representation) of the pure braid group
in $n+2$ strands, which can again be made explicit.  (This was shown
by topological methods in \cite{VenkataramanaTN:2014}, subject to a technical
constraint on the ramification.)  However, it seems natural from a
monodromy perspective to focus on the specific {\em conjugacy classes} of
elements that arise, rather than their actual representations as matrices.

Another important fact about Jordan-Pochhammer representations can be
derived directly from the relation to Pryms.

\begin{prop}
  A Jordan-Pochhammer representation with parameters $(-1;-1,\dots,-1)$
  over a field not of characteristic 2 preserves a unique symplectic form.
\end{prop}

\begin{proof}
  In characteristic 0, the representation acts on the middle homology of a
  hyperelliptic curve, and thus indeed preserves a symplectic form.
  Otherwise, we may lift to characteristic 0, and observe that
  semicontinuity implies that the mod $p$ representation also preserves an
  alternating form.  Uniqueness then follows immediately from
  irreducibility, as does nondegeneracy.
\end{proof}

\begin{prop}
  Let $L/K$ be a quadratic field extension, and let
  $\lambda_0,\dots,\lambda_{n+1}\in L$ be elements of norm 1 multiplying to
  1, none of which are equal to 1.  Then the corresponding
  Jordan-Pochhammer representation in $\GL_n$ fixes a nondegenerate
  anti-Hermitian form, unique up to scalars.
\end{prop}

\begin{proof}
  In light of the previous proposition, we may assume that at least one
  parameter is not $-1$.  Suppose first that $L/K=\C/\R$ and
  $\lambda_0,\dots,\lambda_{n+1}$ are roots of unity, so that the
  Jordan-Pochhammer representation is the monodromy of a Prym.  There is a
  natural $\C$-valued alternating form on $H^1(C;\C)$ given by integrating
  the wedge of differentials over $C$, which combined with the natural
  conjugation gives an anti-Hermitian form, given (up to a constant factor)
  by
  \[
  \langle f_1\,dz+\bar{g}_1\,d\bar{z},f_2\,dz+\bar{f}_2\,d\bar{z}\rangle
  =
  2\sqrt{-1}\iint \bigl(\bar{f_1}g_2-g_1\bar{f_2}\bigr) dxdy.
  \]
  It follows that the monodromy preserves the restriction of this form to
  the appropriate eigenspace of $\mu_N$.  This eigenspace splits naturally
  relative to the form as an orthogonal sum of a space of holomorphic
  differentials and a space of antiholomorphic differentials.  On the
  holomorphic differentials, it is $\sqrt{-1}$ times a positive definite
  form, while on the antiholomorphic differentials it is $\sqrt{-1}$ times
  a negative definite form, and thus it is indeed nondegenerate.
  Uniqueness again follows from irreducibility.

  Since the cases in which the parameters are roots of unity are Zariski
  dense, it follows that the representation over any field fixes at least
  one skew-bilinear form.  Uniqueness up to scalars follows by
  irreducibility, and thus the forms are proportional to their conjugates;
  by Hilbert's Theorem 90, they are proportional to anti-Hermitian forms
  which are unique up to rescaling by $K^\times$.
\end{proof}

\begin{cor}
  For any parameters $\lambda_0,\dots,\lambda_{n+1}\in S^1$ multiplying to
  1 and all different from $1$, the Jordan-Pochhammer representation
  preserves a nondegenerate Hermitian form with
  \[
  -1+\sum_{0\le i\le n+1} \{\log(\lambda_i/2\pi\sqrt{-1})\}
  \]
  positive eigenvalues and
  \[
  -1+\sum_{0\le i\le n+1} \{-\log(\lambda_i/2\pi\sqrt{-1})\}
  \]
  negative eigenvalues.
\end{cor}

\begin{proof}
  Indeed, in the (topologically) dense set of cases when the parameters are
  roots of unity, these are nothing other than the numbers of holomorphic
  and antiholomorphic differentials in the relevant eigenspace.
\end{proof}

Let
\[
R_{JP}:=\Z[\lambda_0,\dots,\lambda_{n+1},1/(\lambda_0-1)\cdots(\lambda_{n+1}-1)]/(\lambda_0\cdots\lambda_{n+1}-1)
\]
be the ring over which the Jordan-Pochhammer representation is generically
defined.

\begin{prop}
  The Jordan-Pochhammer representation can be defined globally over
  $R_{JP}$, and preserves an anti-Hermitian form with coefficients in
  $R_{JP}$ and unit determinant.
\end{prop}

\begin{proof}
  By irreducibility, the family of Jordan-Pochhammer representations is
  a $\PGL_n$-bundle over $\Spec(R)$, and $\im(g_1-1)$ produces a section of
  the corresponding Brauer-Severi variety, so that the representation acts
  on a vector bundle.  But $R$ is open in $\A^{n+1}$, so the vector bundle
  is trivial.

  Similarly, the space of skew-bilinear forms is locally free of rank 1, so
  trivial, and thus there is such a form with unit determinant, itself
  determined up to rescaling by a unit.  In particular, if $\beta$ is such
  a form, then there is a unit $\zeta$ such $\bar\beta = -\zeta \beta$,
  which by uniqueness must have norm $1$.  It is moreover $1$ at the fixed
  points of the involution (where all parameters are $-1$), and the
  involution is tamely ramified, and thus the corresponding cohomology
  class is trivial, letting us rescale $\beta$ to be anti-Hermitian as
  required.
\end{proof}

\begin{rem}
  The form, of course, is only unique up to rescaling by invariant units,
  e.g., $(\lambda_i-1)^2/\lambda_i$.
\end{rem}

\begin{cor}
  Suppose $C/\P^1$ is a cyclic $N$-cover such that none of its ramification
  degrees are prime powers.  Then $\Prym(C)$ is isomorphic to its dual.
\end{cor}

\begin{proof}
   The constraint on the ramification degrees implies that each
   $1-\lambda_i$ is a unit in $\Z[\zeta_N]$, and thus the map from $R_{JP}$
   to $\Q(\zeta_N)$ factors through $\Z[\zeta_N]$.  Since the standard
   polarization of $\Prym(C)$ is monodromy-invariant, it must be a scalar
   multiple (by an element of $\Z[\zeta_N]$) of the unique form coming from
   the Proposition, and thus the kernel of the standard polarization is the
   kernel of an endomorphism, so factors through an isomorphism
   $\Prym(C)\cong \Prym(C)^\vee$.
\end{proof}

\begin{rem}
   This does not imply that $\Prym(C)$ admits a principal polarization,
   which requires that the isomorphism to $\Prym(C)^\vee$ be its own dual
   as an isogeny.
\end{rem}

\section{Largeness}

Let $L/K$ be a quadratic extension of number fields, with corresponding
maximal orders $O_L$, $O_K$, and let $A\mapsto A^\dagger$ denote the
corresponding conjugate transpose operation on $\Mat_n(L)$.  Given an
anti-Hermitian matrix $A=-A^\dagger$ with nonzero determinant, we define
the corresponding unitary group $\GU(A)$ to be the subgroup of
$\GL_n(O_L)$ consisting of matrices $g$ such that $gAg^\dagger=A$.  If
$\frakp$ is any prime of $K$, we similarly define $\GU(A_\frakp)$ to
be the analogous (compact!) group of matrices over the completion at
$\frakp$.  We also define $\SU$ to be the corresponding determinant 1
subgroups.  Let $S_A$ be the set of finite primes of $K$ not dividing
$\det(A)\overline{\det(A)}$.

\begin{defn}
  A homomorphism $\phi:G\to \SU(A)$ is {\em large} if the
  composition
  \[
  G\to \SU(A)\to \prod_{\frakp\in {S_A}} \SU(A_\frakp)
  \]
  has dense image.  A homomorphism $\phi:G\to \GU(A)$ is large if the
  restriction to the preimage of $\SU(A)$ is large.
\end{defn}

\begin{rem}
  For our application to monodromy, we cannot expect density in the
  ad\`elic unitary group, for the simple reason that the determinant is
  always a global root of unity.  Of course, to understand the image of a
  large homomorphism, all that remains is to control the determinants.  For
  our families of covers of $\P^1$, this is straightforward, but for the
  higher genus families our arguments only give a lower bound on the image
  of the determinant.
\end{rem}

The $p$-adic groups $U(A_\frakp)$ come in three flavors, depending
on how $\frakp$ behaves in $L$.  If $\frakp$ is split, then
$\GL_n(O_{L,\frakp})\cong \GL_n(O_{K,\frakp})^2$ and
$A_\frakp$ induces a nondegenerate pairing between the two free
modules, so that $\GU(A_\frakp)\cong \GL_n(O_{K,\frakp})$ and
$\SU(A_\frakp)\cong \SL_n(O_{K,\frakp})$.  Similarly, if
$\frakp$ is inert, then $A_\frakp$ is an anti-Hermitian form
and thus $\SU(A_\frakp)$ is the usual $p$-adic unitary group, a
quasi-split form of $\SU_n$ over $O_{K,\frakp}$.  (An easy induction
by powers of $\frakp$ shows that any two nondegenerate
anti-Hermitian forms are equivalent, so give isomorphic groups.)  In either
case, reduction modulo $\frakp$ gives a surjection to $\SL_n$ or
$\SU_n$ over the corresponding residue field.

When $\frakp$ is tamely ramified (we ignore the wild case, as it
will not arise in our application), the group may still be viewed
as a group scheme over $O_{K,\frakp}$, but is no longer reductive.
If $\frakp'$ is the corresponding prime of $L$, then induction by
powers of $\frakp'$ again shows that $A_\frakp$ is essentially
unique.  In particular, we may take $A_\frakp$ to have coefficients
in $O_{K,\frakp}$, and thus be the base change of a symplectic form.
(In particular, $n$ must be even!)  We thus see that $\SU(A_\frakp)$
contains $\Sp_n(O_{K,{\mathfrak
    p}})$, and thus reduction mod $\frakp'$ gives a surjection from
$\SU(A_\frakp)$ to $\Sp_n(k_\frakp)$.  The kernel, however, has
nontrivial image modulo $\frakp$, an additive group of dimension
$n(n-1)/2-1$, which may be identified with the space of alternating
forms over $k_\frakp$ orthogonal under the trace pairing to $A$.
(To be precise, if $\pi$ is a uniformizer for $\frakp'$ such that
$\bar\pi=-\pi$, then an element $1+\pi B$ is unitary mod $\frakp$
iff $BA = -(BA)^t \pmod \frakp'$ and has determinant 1 iff $\Tr(B)=0$.)

\begin{rem}
  One could also consider the case of a Hermitian form, for which the
  unramified cases agree (scale by an element of trace 0) while the
  ramified case (again assuming tameness) would have reduction an extension
  of $\SO_n$ by an $n(n+1)/2-1$-dimensional space of symmetric forms.
  Also, note that although we work with a matrix $A$, this is merely to
  simplify the discussion; everything below works equally well for forms on
  locally free $O_K$-modules.
\end{rem}

Since largeness is a question of (topological) surjectivity to a product,
it will be helpful to have some lemmas about such surjectivity.  One
simplification is that apart from some low rank cases that we avoid, the
factors in the product are all perfect.  (This also adds the simplification
that we can restrict our homomorphism to $\GU(A)$ to the derived subgroup of
the domain without affecting largeness.)

\begin{lem}
  Suppose $H\subset G_1\times G_2$ is a proper subgroup such that both
  projections are surjective.  Then there are proper normal subgroups
  $N_1\normal G_1$, $N_2\normal G_2$ such that $G_1/N_1\cong G_2/N_2$.
\end{lem}

\begin{proof}
  Let $N_1\subsetneq G_1$ be the kernel of the projection $H\to G_2$ and
  similarly for $N_2\subsetneq G_2$.  Then $N_1\times N_2\normal H$ and thus
  we may quotient to obtain a subgroup $H/(N_1\times N_2)\subset
  (G_1/N_1)\times (G_2/N_2)$ for which both projections are also injective.
  But then both projections are isomorphisms, and the claim follows.
\end{proof}

\begin{lem}
  Let $G_1,\dots,G_n$ be a sequence of perfect groups and let $H\subset
  G_1\times\cdots\times G_n$ be a subgroup such that any projection
  $\pi_{ij}:H\to G_i\times G_j$ is surjective.  Then
  $H=G_1\times\cdots\times G_n$.
\end{lem}

\begin{proof}
  Let $N_i$ be the kernel of the projection $H\to G_i$.  This is normal in
  $H$ and surjects on $G_j$ for all $i\ne j$.  For any $g_1,h_1\in G_1$ and
  $i,j>1$, there is an element of $N_i$ with first coordinate $g_1$ and an
  element of $N_j$ with first coordinate $h_1$, and thus the commutator is
  an element of $N_i\cap N_j$ with first coordinate
  $g_1h_1g_1^{-1}h_1^{-1}$.  In particular, we see that $N_i\cap N_j$
  surjects on $G_1$ for all $1\notin \{i,j\}$, and inductively that
  $N_2\cap\cdots N_n$ surjects on (so contains!) $G_1$.  Since by induction
  $H$ surjects on $G_2\times\cdots\times G_n$, the claim follows.
\end{proof}

Note that the above arguments work equally well for closed subgroups of
topological groups and extend immediately to infinite products.  We thus
see that as long as the factors $\SU_n(A_\frakp)$ are perfect (automatic if
$n\ge 4$ by Corollary \ref{cor:sc_is_perfect} below), showing largeness
reduces to showing density of the image in any product over a {\em pair} of
primes.  Each factor is an extension of a finite simple group by a
pro-$p$-group and (since it is perfect) no quotient is a pro-$p$-group.  It
follows (again avoiding low-rank cases) that for primes with distinct
residue fields, surjectivity to the pair reduces to individual surjectivity
along with surjectivity to the product of finite simple groups!  (If the
proper subgroup $H$ has surjective projections, then quotienting by the
full pro-$p$-groups gives an isomorphism between the finite simple
quotients.)  We thus reduce to showing (a) surjectivity to each factor and
(b) surjectivity to each pair of finite simple groups over the same residue
field.  (In our case, $K$ will be an abelian extension of $\Q$, so ``same
residue field'' simply means ``same characteristic''.)  In practice, both
(a) and (b) start with a common first step, namely showing that the map to
each residue group is surjective.  Then (a) (a.k.a. ``lifting'') reduces to
showing that the image contains the kernel of reduction; since that group
is nilpotent, this inductively reduces to {\em linear} questions.  For (b),
on the other hand, it suffices to rule out the image being the graph of an
isomorphism.

\section{Largeness mod primes}

The first step in proving that the (reduced, cyclotomic) Jordan-Pochhammer
representation is large is to show that it is essentially surjective modulo
primes of the real subfield.  (This was shown for $n>8$ in
\cite{VolkleinH:1998}; we extend this to $n>4$ and describe the limited
exceptions for $n=3,4$ as well as the less limited exceptions for $n=2$.)
This in turn reduces to understanding the images of {\em modular}
Jordan-Pochhammer representations, i.e., in which the parameters are
elements of $\F_q^\times\setminus\{1\}$.

Since the generators of the free group map to pseudo-reflections
under the Jordan-Pochhammer representation, the image will be an
irreducible subgroup of $\GL_n(\bar\F_p)$ generated by pseudo-reflections.
The classification of such groups up to conjugation is well-known, though
oddly does not appear to have been stated in any single place.  These
groups fall into three categories:
\begin{itemize}
  \item[1.] Classical groups satisfying one of $G=\Sp_{2n}(\F_q)$,
    \begin{align}
      \SL_n(\F_q)\subseteq &G\subseteq \GL_n(\F_q),\\
      \SU_n(\F_q)\subseteq &G\subseteq \GU_n(\F_q),\\
      \SL_2(\F_q)\subseteq &G\subseteq \F_{q^2}^\times\otimes \GL_2(\F_q),\\
      \Omega_n(\F_q)\subseteq &G\subseteq \GO_n(\F_q),
    \end{align}
    with the last case excluding the groups $\Omega_n(\F_q)$, $\SO_n(\F_q)$
    not containing reflections.
  \item[2.] Faithful reductions of irreducible {\em complex} reflection
    groups (as classified in \cite{ShephardGC/ToddJA:1954}).
  \item[3.] The groups $S_{pn}\subset \GL_{pn-2}(\F_p)$,
    $3.\Alt_6\subset\GL_3(\F_4)$, $3.\Alt_7\times 2\subset \GU_3(\F_5)$, 
    $4.\PSL_3(\F_4).2\subset\GU_4(\F_3)$, and
    $3.\SU_4(\F_3).2_2\subset\GU_6(\F_2)$.
\end{itemize}
For the classification in dimension $>2$, see
\cite[\S2.1]{MalleG/MatzatBH:2018} and references therein.  For $n=2$, it
was shown in \cite{DicksonLE:1958} that any finite irreducible subgroup of
$\PGL_2(\bar\F_p)$ either lifts to characteristic 0 or is $\PSL_2(\F_q)$ or
$\PGL_2(\F_q)$ for some $q$.  Since any nonidentity element of
$\PGL_2(\bar\F_q)$ is the image of a transvection or a pair of homologies,
it follows easily that any reflection group with liftable image is also
liftable, and that the lifts of $\PSL_2(\F_q)$ and $\PGL_2(\F_q)$ are of
the form stated.

\begin{rem}
  The groups $3.\Alt_6$ and $3.\SU_4(\F_3).2_2$ appearing in case 3 are the
  reductions mod 2 of $ST_{27}$ and $ST_{34}$ respectively; in both cases
  the kernel of reduction is $\mu_2$.  In all other cases for which the
  reduction of a Shephard-Todd group is not faithful, the reduction is
  either reducible or appears elsewhere on the list.
\end{rem}

The groups of type 2 reduce to questions in characteristic 0.

\begin{lem}
  For $n>4$, let $\lambda_0,\dots,\lambda_{n+1}\in \F_q\setminus \{0,1\}$
  be parameters multiplying to $1$.  Then the image of the
  Jordan-Pochhammer representation cannot be a faithful reduction of an
  irreducible complex reflection group.
\end{lem}

\begin{proof}
  If the image of an irreducible Jordan-Pochhammer representation is the
  faithful reduction of a complex reflection group, then the preimages of
  the generators are still pseudoreflections, and thus the representation
  is the reduction of a representation in characteristic 0.  The cases with
  finite monodromy were classified in \cite{TakanoK/BannaiE:1976}, and
  there are no instances with $n>4$.  (They worked with a slightly
  different version of the monodromy in which the scalar is absorbed in one
  of the generators, but this has no effect on finiteness.)
\end{proof}

\begin{rem}
  Up to the braid group action and Galois, the only cases with $n\in
  \{3,4\}$ have parameters
  \[
  (\zeta_6^2;\zeta_6,\zeta_6,\zeta_6,\zeta_6)\ [3^{1+2}.2],
  (\zeta_6;\zeta_6^2,\zeta_6,\zeta_6,\zeta_6)\ [\ST_{26}]
  \]
  for $n=3$ ($p\ne 3$) or
  \[
  (\zeta_6;\zeta_6,\zeta_6,\zeta_6,\zeta_6,\zeta_6)\ [\ST_{32}]
  \]
  for $n=4$.  For $n=2$, there are a large number of cases: any pair of
  elements that generate a finite subgroup of $\PSL_2(\C)$ gives
  rise to 8 finite Jordan-Pochhammer representations over $\C$.  (Two ways
  each of lifting the generators to pseudoreflections, plus an additional
  choice of which eigenvalue of the product is $\lambda_0$.)
\end{rem}

\begin{rem}
  The two cases for $n=3$ correspond to covers with the same ramification,
  but with a different choice of which ramification point is being moved.
  If we allow {\em all} of the ramification points to move, the group
  becomes larger, but is still finite, as it contains the
  single-moving-point monodromy as a normal subgroup.
\end{rem}

For the groups of type 3 and the orthogonal groups, the pseudoreflections
all have order 2, forcing the parameters to have the form
\[
(\lambda_0;-1/\lambda_0,\dots,-1/\lambda_0)
\]
with $\lambda_0^n=(-1)^{n+1}$ but $\lambda_0\notin \{\pm 1\}$.  The
symmetric and orthogonal cases are thus immediately ruled out by having
centers of order $\le 2$.  For the remaining sporadic cases, the rank 3
cases are ruled out by the fact that the corresponding parameters in
characteristic 0 generate an imprimitive reflection group, while
$4.\PSL_3(\F_4).2$ would need to have a central element of order 4.
Finally, if the image were $2.U_4(3).2_2\cong \ST_{34}/\mu_2$, then
we could lift the generators to pseudoreflections in $\ST_{34}$; the lifts
would still multiply to a central element and thus give a contradiction to
\cite{TakanoK/BannaiE:1976}.

For the $\Sp_{2n}$ case, all of the generators are transvections, and thus
the parameters must be
\[
(\lambda_0;1/\lambda_0,\dots,1/\lambda_0)
\]
Moreover, since $\Sp_{2n}$ has center $\mu_2$, we must have $\lambda_0=-1$,
and thus the parameters are all $-1$, in which case the group is indeed
contained in $\Sp_{2n}(\F_p)$, and thus (by the above arguments) equal to
$\Sp_{2n}(\F_p)$.

It remains only to show that in the remaining cases, the group
contains the relevant linear or unitary group.  (I.e., we must rule
out that it can be defined over a subfield.) The argument for this in
\cite{VolkleinH:1998} remains valid for $n>2$, and indeed we can prove a
generalization, which will be useful for dealing with pairwise
largeness.

\begin{lem}\label{lem:pairwise}
  Let $L/K$ be a quadratic extension of finite \'etale algebras, and for
  $n>2$, let $\lambda_0,\dots,\lambda_{n+1}\subset L$ be elements of norm 1
  multiplying to 1 and such that for each $i$, $1-\lambda_i$ is a unit.
  If $L$ is generated by $\lambda_0,\dots,\lambda_{n+1}$, then there is no
  subalgebra $L'\subset L$ such that the image of the Jordan-Pochhammer
  representation in $\GL_n(L)$ is contained in a conjugate of
  $\GL_1(L)\otimes \GL_n(L')$.
\end{lem}

\begin{proof}
  Otherwise, after making the appropriate choice of basis, we may assume
  that the generators are all proportional to matrices with coefficients in
  $L'$.  But then the generators must actually have coefficients in $L'$,
  since the conjugates under $\Gal(L/L')$ are proportional and thus equal (since $1$ has
  different multiplicity than any other eigenvalue).  In particular, the
  nontrivial eigenvalues of the generators must be in $L'$, i.e.,
  $\lambda_0\lambda_i\in L'$ for all $i$.  But the product of the
  generators is also in $\GL_n(L')$ and thus $\lambda_0\in L'$, forcing
  $L'=L$.
\end{proof}

\begin{rem}
  In particular, given any finite \'etale algebra $R$ and any parameters
  $\lambda_0,\dots,\lambda_{n+1}\in R$ multiplying to 1 and such that
  $\lambda_i-1$ is a unit for all $i$, we may take $L$ to be the subalgebra
  of $R\times R$ generated by
  $(\lambda_0,1/\lambda_0)$,\dots,$(\lambda_{n+1},1/\lambda_{n+1})$, and
  find that the Jordan-Pochhammer representation is contained in
  $\GU_n(L/(L\cap R))$ and in no conjugate of a smaller unitary or linear group.
\end{rem}

\begin{rem}
  For $n=2$, this argument fails, as we can no longer reliably distinguish
  the two eigenvalues of the generators.  We instead find that the group is
  contained in a conjugate of $\GL_1(L)\SL_2(K')$ where $K'$ is generated
  by $\lambda_0+1/\lambda_0$, $\lambda_1+1/\lambda_1$,
  $\lambda_2+1/\lambda_2$ and
  \[
  \frac{(1-\lambda_0)(1-\lambda_1)(1-\lambda_2)(1-\lambda_0\lambda_1\lambda_2)}{\lambda_0\lambda_1\lambda_2}.
  \]
\end{rem}

\section{Pairwise surjectivity}

The next step in proving largeness of cyclotomic Jordan-Pochhammer
representations is to show that for any pair of (irreducible) primes of
$\Q(\zeta_N+\zeta_N^{-1})$ over the same rational prime $p$, the reduction
map remains surjective.  Since cyclotomic fields are abelian extensions of
$\Q$, the quotients of $\Z[\zeta_N]$ by the two primes are isomorphic,
either $\F_{q^2}$ or $\F_q\times \F_q$.  Fix a standard form of this
\'etale algebra $A$ and fix choices of morphisms $\Z[\zeta_N]\to A$ with
kernels the two primes, so that we get a pair of representations of
$F_{n+1}$ in $\GU_n(A/\F_q)$ (which is either $\GL_n(\F_q)$ or
$\GU_n(\F_q)$).  These two morphisms differ by an automorphism of
$\Q(\zeta_N)$, which may be denoted by an element $l\in (\Z/N\Z)^\times$,
where we note that two elements correspond to the same prime in
$\Q(\zeta_N+\zeta_N^{-1})$ iff they are in the same coset of the subgroup
generated by $p$ and $-1$.

Our assumption $n>4$ ensures that $\GU_n(A)'/Z(\GU_n(A)')$ is a simple group,
and thus if $G$ is the image of $F_{n+1}$ in $\GU_n(A)$, then the image of
$G'$ in $\PSU_n(A/\F_q)^2$ is either everything or the graph of an
automorphism.  The automorphisms of $\PSU_n(A/\F_q)$ (indeed, the
automorphisms of any simple group of Lie type) are well-understood: they
are generated over $\PGU_n(A/\F_q)$ by the action of $\Gal(A/\F_p)$.  (Note
that in the $\SL$ case, the usual diagram automorphism corresponds to the
automorphism of $A$ that swaps the two factors of $\F_q$.)  In particular,
we see that we may absorb the action of $\Gal(A/\F_p)$ into our choice of
$l$ above, and thus reduce to the case that the automorphism of
$\PSU_n(A/\F_q)$ comes from $\PGU_n(A/\F_q)$.

Now, consider more generally the image of $g\in F_{n+1}$ in
$\PGU_n(A/\F_q)^2$, say $(g_1,g_2)$. For any $h\in F_{n+1}'$, its image in
$\PSU_n(A/\F_q)^2$ is $(\phi(h),u\phi(h)u^{-1})$ for some fixed element
$u\in \PGU_n(A/\F_q)$, and this must respect the action of conjugation:
\begin{align}
g_1 \phi(h) g_1^{-1} &= \phi(g h g^{-1})\\
g_2 u \phi(h) u^{-1} g_2^{-1} &= u \phi(g h g^{-1}) u^{-1}
\end{align}
Since this is true for all $h$, we conclude that
\[
g_2 = u g_2 u^{-1},
\]
and thus the image of $F_{n+1}$ in $G^2\subset \PGU_n(A/\F_q)^2$ is still
the graph of the automorphism corresponding to $u$.

In particular, we see that the image of $F_{n+1}$ in $\GU_n(A/\F_q)^2$ is
contained in a conjugate of $\GU_n(A/\F_q)$.  By Lemma \ref{lem:pairwise},
this implies that the diagonal copy of $A$ contains
$(\lambda_i,\lambda'_i)$ for all $i$, and thus that $\lambda'_i=\lambda_i$
for all $i$.  In other words, the two primes are actually the same!

We have thus proved the requisite pairwise surjectivity.

\begin{rem}
Although we have assumed $n>4$ here, the argument actually works for $n>2$
(assuming, that is, that both maps are surjective and the group is not
$\SU_3(\F_2)$).
\end{rem}

\section{Lifting to the $p$-adics}

\subsection{Graded Lie algebras from $p$-adic groups}

Let $R$ be a mixed characteristic dvr with maximal ideal $\frakm$ and
residue field $k$ of characteristic $p$.  Then $\GL_n(R)$ has a natural
filtration coming from the filtration of $R$ by powers of $\frakm$; that
is, we have an inverse system of groups
\[
\cdots\to \GL_n(R/\frakm^e)\to\cdots \GL_n(R/\frakm^2)\to \GL_n(k)\to
1.
\]
If $k$ has characteristic $p$, then $(p)=\frakm^d$ for some $d$, and we
define for $e>0$, $I_{e/d}:=\ker(\GL_n(R)\to \GL_n(R/\frakm^e))$, with
$I_0=\GL_n(R)$.  For $e>0$, there is an essentially natural identification
of $I_{e/d}/I_{(e+1)/d}$ with $\gl_n(k)$ (up to a choice of
uniformizer), and this glues together to form a graded Lie algebra with an
action of $\GL_n(k)$.  More precisely, consider a commutator
$(A,B)=A^{-1}B^{-1}AB$ where $A\in I_{e_1/d}$, $B\in I_{e_2/d}$.  We may
rewrite this as
\[
(A,B) = 1 + A^{-1}B^{-1}[A,B] = 1 + A^{-1}B^{-1}[A-1,B-1]
\]
and thus find that $(A,B)\in I_{(e_1+e_2)/d}$ and the corresponding map on
the associated graded is just the usual Lie bracket.  This graded Lie
algebra has an additional structure coming from the $p$-th power map on
$\GL_n$: if $A\in I_{e/d}$, then $A^p\in I_{\min(pe/d,(e/d)+1)}$.
The precise action may be seen by noting that the first and last terms in
the sum
\[
A^p-1 = \sum_{1\le i\le p} \binom{p}{i} (A-1)^p
\]
strictly dominate, so that
\[
A^p-1 = \begin{cases}
  (A-1)^p +o(p^{pe/d})& e/d<1/(p-1)\\
  p(A-1)+(A-1)^p +o(p^{p/(p-1)}) & e/d=1/(p-1)\\
  p(A-1)+ o(p^{1+e/d}) & e/d >1/(p-1).
\end{cases}
\]

If $R$ is complete and $G\subset \GL_n(R)$ is a {\em topologically} closed
subgroup, then there is an induced filtration on $G$, with $G/(G\cap
I_{1/d})$ acting on an associated graded Lie subalgebra (over $\F_p$ not
$k$!)
\[
\bigoplus_{i>0} (G\cap I_{i/d})/(G\cap I_{(i+1)/d})
\subset
\bigoplus_{i>0} \gl_n(k),
\]
which is also closed under the $p$-th power operation.  When $G$ is a {\em
  Zariski} closed subgroup, we see that this associated graded Lie algebra
is nothing other than the obvious graded Lie algebra structure with
$\Lie(G)(k)$ in every positive degree, so is independent (as one would
expect) of the choice of faithful representation.  (Also, the $p$-th power
structure has the form one would expect, with $(A-1)^p$ computed using the
restricted Lie algebra structure on $\Lie(G)$.)

In each case of interest, the problem of showing that the monodromy is
large reduces to showing that some subgroup $H$ of the target group $G$
(which is always closed in $\GL_n(R)$, and often Zariski closed) is equal to $G$,
which in turn reduces to showing that it surjects on $\GL_n(k)$ and has the
same graded Lie algebra.  The benefit of this is that the graded Lie
algebra is often generated in degree $1/d$.

\begin{prop}
  Let $G/k$ be a simply connected simple affine group scheme, and suppose
  either that $2\in k^\times$ or $G$ is not of the form $\Sp_{2n}$.  Then
  $[\Lie(G)(k),\Lie(G)(k)]=\Lie(G)(k)$.
\end{prop}

We may also use the action of $G(k)$.

\begin{prop}
  Let $G/k$ be a simply connected simple affine group scheme, and suppose
  that $G\ne \SL_2/\F_2$.  Then $H_0(G(k)';\Lie(G)(k))=0$.
\end{prop}

\subsection{Unramified lifts in general}

It turns out (following \cite{VasiuA:2003}) that in the unramified case, lifting
tends to be automatic.  More precisely, apart from a small number of
low-rank cases, surjectivity of $H\to G_0$ implies that the corresponding
graded Lie algebra is also large.

Given a reductive group over an unramified extension of $\Z_p$, we may
always take the restriction of scalars to obtain a reductive group over
$\Z_p$; if the original group is a twist of a group of type $T$, then the
new group will be a twist of $T^m$ where $m$ is the degree of the
extension.  With this in mind, we denote (for $X=A,\dots,G$) by
${}^d\!X^{(m)}_n$ the twist of $X^m_n$ by the diagram automorphism that cycles
the $m$ copies of $X$, twisted by the diagram automorphism of $X$, so that
${}^d\!X^{(m)}_n(\F_p)\cong {}^d\!X_n(\F_{p^m})$.

The main result is then the following.  Note that $G(\Z_p)'$ is not the
same as the group of $\Z_p$ points of the derived subgroup {\em scheme}
$G'$.

\begin{thm}\label{thm:vasiu_fixed}
  Let $G$ be a Zariski-connected reductive group over $\Z_p$ such that
  one of the following holds:
  \begin{itemize}
  \item[(i)] $p>3$.
  \item[(ii)] $p=3$ and $\Ad(G)$ has no factor of type $A_1$.
  \item[(iii)] $p=2$ and $\Ad(G)$ has no factor of type $A_1$,
    $A_1^{(2)}$, $A_2$, ${}^2\!A_2$, $B_2$, $G_2$, $A_3$, ${}^2\!A_3$, $B_3$,
    or $D_4$.
  \end{itemize}
  Then for any closed subgroup $H\subset G(\Z_p)$, if $H\bmod p$ contains
  $G(\F_p)'$, then $H$ contains $G(\Z_p)'$.
\end{thm}

\begin{rem}
  The exceptions include all of the cases ($A_1(2)$, $A_1(3)$,
  ${}^2\!A_2(2)$, $B_2(2)$, $G_2(2)$) for which the corresponding simply
  connected group of Lie type is not perfect.  In particular, the
  conclusion can be strengthened to say that $H'=G(\Z_p)'$.
\end{rem}

\begin{rem}
  This is essentially a reformulation of the main result of
  \cite{VasiuA:2003}, except that the argument there is not entirely valid.
  Indeed, the statement in \cite{VasiuA:2003} is not correct, as it misses
  some counterexamples: $W(E_8)\subset O_8^+(\Z_2)$, $W(E_7)\subset
  O_7(\Z_2)$, $S_6\subset O_5(\Z_2)$, and the preimage of $S_8\subset
  O^+_6(\Z/4\Z)$.  (The last example is constructed as an invariant
  subquotient of $(\Z/4\Z)^8$ coming from the $\Z/8\Z$-valued form $\sum_i
  x_i^2$ on $(\Z/4\Z)^8$; the induced form on the subquotient is even, so
  dividing by 2 gives a quadratic form over $\Z/4\Z$.)
\end{rem}

\begin{rem}
  In the present paper, we of course only need cases in which every factor
  has type ${}^d\!A^{(m)}_n$ with $n>4$, but we will need the $D_n$ case in
  future work on Selmer groups, and the general reductive case makes the
  inductive argument simpler.
\end{rem}

This reduces in several stages to the case of a simply connected simple
group.  The first is that we may replace $G$ by $\Ad(G)$.

\begin{lem}
  If $G$ satisfies the hypotheses and $\Ad(G)$ satisfies the conclusion, then
  $G$ also satisfies the conclusion.
\end{lem}

\begin{proof}
  Suppose $H\subset G(\Z_p)$ were a counterexample.  Then $H'\subset
  G(\Z_p)'\subset G'(\Z_p)$ still surjects on $G(\F_p)'$ since $G(\F_p)'$
  is perfect.  Thus $H'$ would also give a counterexample, and thus we may
  assume $G=G'$.  Since the image of $H$ in $\Ad(G)(\Z_p)$ surjects on
  $\Ad(G)(\F_p)'$, the image of $H$ contains $\Ad(G)(\Z_p)'$ and thus
  contains the image of $\tilde{G}(\Z_p)$ in $\Ad(G)(\Z_p)$.  In
  particular, it contains the image of $G(\Z_p)'$ in $\Ad(G)(\Z_p)$, so
  that $\langle H,Z(G)(\Z_p)\rangle$ contains $G(\Z_p)'$.  The latter group
  is perfect and thus
  \[
  G(\Z_p)'\subset \langle H,Z(G)(\Z_p)\rangle'\subset H
  \]
  as required.
\end{proof}

\begin{lem}
  Let $G$ be an adjoint group satisfying the hypotheses such that
  the conclusion holds for $\tilde{G}$.  Then it holds for $G$.
\end{lem}

\begin{proof}
  Again, suppose $H\subset G(\Z_p)$ were a counterexample.  We may assume
  $H$ is perfect, so that $H\subset G(\Z_p)'$.  But this implies that $H$
  is contained in the image of $\tilde{G}(\Z_p)$, so we may take its
  preimage in $\tilde{G}(\Z_p)$.  But this is a proper closed subgroup
  surjecting on $\tilde{G}(\F_p)$, so would give a counterexample for
  $\tilde{G}$.
\end{proof}

The following is essentially the same argument we used to reduce largeness
to surjectivity modulo pairs of primes and lifting for a single prime.

\begin{lem}\label{lem:lifting_products}
  Suppose the theorem holds for every simply connected simple group over
  $\Z_p$ which is not one of the listed exceptions.  Then it holds for any
  (finite) product of such groups.
\end{lem}

\begin{proof}
  Let $H\subset \prod_i G_i$ be a counterexample.  Then the image of $H$ in
  each factor $G_i$ contains $G_i(\Z_p)$, and thus the associated graded
  Lie algebra of $H$ surjects onto the associated graded Lie algebra of
  each $G_i$.  Since $H$ surjects on $\prod_i G_i(\F_p)$, the associated
  graded contains the image of $\Ad(g_i)-1$ for all $g_i\in G_i(\F_p)$, which by
  Corollary \ref{cor:sc_is_perfect} below implies that the associated
  graded of $H$ contains the associated graded of each $G_i$, so that $H$
  is not actually a counterexample.
\end{proof}

At this point, we restrict our attention to the case of a simply connected
simple group which is not one of the excluded types.  It turns out that the
main further reductions of \cite{VasiuA:2003} apply.

\begin{prop}
  Let $G$ be any simply connected simple group.  The image of a proper
  closed subgroup $H\subset G(\Z_p)$ in $G(\Z/p^2\Z)$ is a proper subgroup.
\end{prop}

\begin{proof}
  If $H$ surjects on $G(\Z/p^2\Z)$, then the degree 1 component of its
  associated graded Lie algebra is $\Lie(G)$.  For $p>2$, the $p$-th power
  operation on that algebra induces an isomorphism between the components
  in consecutive degrees (it simply multiplies $g-1$ by $p$) and thus $H$
  has associated graded Lie algebra $\Lie(G)$ in every degree, so that
  $H=G(\Z_p)$.  This same argument works for $p=2$ if we can show
  surjectivity to $G(\Z/8\Z)$.  The squaring map in degree $1$ has the form
  $q:x\mapsto x+x^2$, and thus in particular is the identity on root
  elements.  Since $G$ is simply connected, $[\Lie(G),\Lie(G)]$ contains
  the Cartan, and thus we conclude that the Lie algebra in degree 2
  contains both the Cartan and the root spaces, so is everything.
\end{proof}

We thus reduce to considering subgroups of $G(\Z/p^2\Z)$ that surject on
$G(\F_p)$.  If $H$ were such a subgroup, then it would meet $\Lie(G)(\F_p)$
in a $G(\F_p)$-invariant subspace, and our objective is to prove that the
subspace cannot be proper.  Since a proper subspace is contained in a
maximal subspace, it suffices to show that $H$ contains an element outside
every maximal subspace, or equivalently that the image of $H\cap
\Lie(G)(\F_p)$ in $\cosoc\Lie(G)(\F_p)=\Lie(G)(\F_p)/\rad(\Lie(G)(\F_p))$
intersects every simple summand.  This is simplified greatly by the fact
that there is a unique such summand, or equivalently that there is a unique
maximal submodule, the restriction to $\F_p$ of the unique maximal
submodule of $G(\bar\F_p)$ acting on $\Lie(G)(\bar\F_p)$.  This follows
from the analogous statement for group schemes via the following.

\begin{lem}
  For any simply connected simple group scheme $G/\F_q$ other than
  $A_1(2)$, every $G(\F_q)$-invariant submodule of $\Lie(G)$ is
  $G(\bar\F_q)$-invariant.
\end{lem}

\begin{proof}(Sketch)
  We in fact show something slightly stronger, namely that
  $\Ad(\bar\F_q[G(\F_q)])=\Ad(\bar\F_q[G])$.  We may feel free to assume
  $G(\F_q)$ simple, as this only fails in four cases ($A_1(3)$, $B_2(2)$,
  ${}^2A_2(2)$, $G_2(2)$) for which the claim may be checked directly.  In
  that case, we observe that it suffices to find {\em some} nontrivial
  smooth connected subgroup scheme $H\subset G$ such that
  $\Ad_G(\bar\F_q[H(\F_q)])=\Ad_G(\bar\F_q[H])$.  Indeed, we could then
  replace $G(\F_q)$ with the group scheme $H^+$ generated by $G(\F_q)$ and
  $H$ without changing the image of the group algebra, and since
  $(H^+)^0(\F_q)$ is a nontrivial normal subgroup of $G(\F_q)$, we see that
  $(H^+)^0(\F_q)=G(\F_q)$, forcing $H^+=G$.

  For instance, let $H$ be the unipotent subgroup corresponding to the
  highest root.  By weight considerations, we see that $\bar\F_q[H]$ is
  $3$-dimensional, and has a basis such that the coordinates of an element
  $\alpha\in \G_a\cong H$ are $(1,\alpha,\alpha^2)$.  It follows that any
  three distinct elements of $H(\bar\F_q)$ span, and thus $H(\F_q)$ spans
  as long as $q>2$.

  For the $q=2$ cases, we may instead take $H$ to be a suitably twisted
  maximal torus.  Such a torus is determined by an element of the
  appropriate coset in the extended Weyl group, and we observe that
  $\bar\F_2[T^w(\F_2)]=\bar\F_2[T^w]$ iff any two distinct eigenspaces of
  $T^w$ have distinct eigenvalues for some element of $T^w(\F_2)$, or
  equivalently if the subgroup of $T(\bar\F_2)$ satisfying $wtw^{-1}=t^2$
  distinguishes root spaces from each other and from the Cartan.  In
  particular, if $w=-1$, i.e., the element of the extended Weyl group that
  negates every root, then this subgroup is $T[3]$ and the condition holds
  as long as no root or difference of roots is 3 times a weight.  This
  condition is invariant under the Weyl group, so we may restrict one of
  the roots to a simple root, at which point it is straightforward to
  check.  In particular, we see that this works for ${}^2A_n$, $n>2$,
  $B_n$, $C_n$, $D_{2n}$, ${}^2D_{2n+1}$, ${}^2E_6$, $E_7$, $E_8$, and
  $F_4$.  (This is every case containing $T^{-1}$ except for ${}^2A_2$ and
  $G_2$, which are not simple, so were already checked directly.)

  For $A_n$, $n>1$, we twist by an $(n+1)$-cycle in $S_{n+1}$, while for
  ${}^2D_{2n}$ and $D_{2n+1}$ we twist by a signed permutation that negates
  $n-2$ elements and acts on the remaining two by $(x,y)\mapsto (y,-x)$.
  Finally, for ${}^3D_4$ and $E_6$, we twist by an (extended) Weyl group
  element with minimal polynomial $w^2+w+1$.
\end{proof}

\begin{rems}
  The case $A_1(2)$ needs to be excluded, as ${\mathfrak{sl}}_2$ splits as
  a direct sum of $1$-dimensional and $2$-dimensional $\SL_2(2)$-modules.
\end{rems}

\begin{rems}
  For $A_1(3)$, the main argument actually works despite the failure to be
  simple, since $\SL_2(\F_3)$ is generated by unipotent elements.  For
  $B_2(2)$, the normal subgroup generated by $H(\F_2)$ is $\Alt_6\subset
  S_6=B_2(2)$, but $\Alt_6$ is not the group of points of a connected
  reductive group, so again $(H^+)^0$ is everything and the argument works.
  Finally, for ${}^2A_2(2)$ and $G_2(2)$, it suffices to verify (by
  standard computational techniques) that the adjoint representation is
  irreducible.
\end{rems}
  
\begin{rems}
  That $\cosoc\Lie(G(\F_q))$ is irreducible is essentially Corollary 3.12
  of \cite{VasiuA:2003}, although the statement there does not explicitly mention
  the exception $A_1(2)$.
\end{rems}
  
\begin{cor}\label{cor:sc_is_perfect}
  For any simply connected simple $G$ not of type $A_1(2)$, $G(\Z_p)'$
  is the preimage of $G(\F_p)'$.
\end{cor}

\begin{proof}
  Since $G(\Z_p)'$ is the common kernel of the $1$-dimensional characters
  of $G(\Z_p)$, it suffices to show that every $1$-dimensional character of
  $G(\Z_p)$ is trivial on the kernel of reduction.  Since the kernel of
  reduction is a $p$-group, if there were a counterexample, there would be
  one of order $p$, which would in turn induce a nontrivial
  $G(\Z_p)$-invariant {\em homogeneous} linear functional on the associated
  graded Lie algebra, and thus a $G(\F_p)$-invariant linear functional on
  $\Lie(G)(\F_p)$.  Such a functional would factor through
  $\cosoc\Lie(G)(\F_p)$, and thus by irreducibility of the latter could
  only exist if $\cosoc\Lie(G)(\F_p)$, and thus $\cosoc\Lie(G)$, were the
  trivial representation.  Since the map $\Lie(G)\to \cosoc\Lie(G)$ is
  injective on long root spaces, the Weyl group acts faithfully, and thus
  $G$ cannot act trivially.
\end{proof}

\begin{rems}
  The claim fails for $\SL_2(\Z_2)$, since $\SL_2(\Z_2)'$ has index 4.
\end{rems}

\begin{rems}
  Note that the graded Lie algebra of $G'(\Z_p)$ is generated by
  commutators of elements that commute mod $p$, and any such element
  is a sum of elements of the form $(\Ad(g)-1)l$, establishing the
  missing step of the proof of Lemma \ref{lem:lifting_products}
  above.
\end{rems}

The following is a simplified (and somewhat generalized: we do not require
the subgroup to be contained in a split torus) rephrasing of a
reduction of Vasiu.  Note that $H$ will typically {\em not} be a simply
connected simple group, but will at least be reductive.

\begin{prop}
  Let $G$ be simply connected and simple, let $\Delta\subset \Aut(G)$ be a
  diagonalizable subgroup scheme.  If $(G^\Delta)^0$ satisfies the
  hypotheses of the theorem, $[\Lie(G^\Delta),\Lie(G^\Delta)]$ is not
  contained in $\rad \Lie(G)$, and the conclusion of the theorem holds for
  $(G^\Delta)^0$, then it holds for $G$.
\end{prop}

\begin{proof}
  Note first that since $G$ is simply connected, the centralizer of any
  diagonalizable group of automorphisms is connected reductive, so that it
  makes sense to ask whether the theorem holds for $G^\Delta$.  Now,
  suppose $H\subset G(\Z/p^2\Z)$ were a counterexample, and let $H_\Delta$
  be the preimage of $G^\Delta(\F_p)\subset G(\F_p)$ inside $G(\Z/p^2\Z)$.
  Since $\Delta$ is diagonalizable, the natural inclusion
  \[
  \Lie(G^\Delta)=\Lie(G)^\Delta\subset \Lie(G)
  \]
  is canonically split.  The kernel $K$ of the splitting (i.e., the sum of
  nontrivial isotypic subspaces for $\Delta$) is normal in $H_\Delta$, and
  thus we may take the quotient $(H+K)/K$ to obtain a subgroup of
  $G^\Delta(\Z/p^2\Z)$ that surjects on $G^\Delta(\F_p)$.  Since lifting
  holds for $G^\Delta$, it follows that
  \[
  H\cap \Lie(G)(\F_p) + K
  =
  H_\Delta\cap \Lie(G)(\F_p) + K
  \supset
  [\Lie(G^\Delta),\Lie(G^\Delta)](\F_p)+K
  \]
  Since $[\Lie(G^\Delta),\Lie(G^\Delta)]$ is not contained in the radical,
  \[
  \rad(\Lie(G))+K\ne \Lie(G),
  \]
  and thus $H\cap \Lie(G)(\F_p)$ is not contained in $\rad\Lie(G)$,
  forcing $H=G(\Z/p^2\Z)$.
\end{proof}

\begin{rems}
  The general diagonalizable subgroup of $\Aut(G)$ fits into a short exact
  sequence
  \[
  1\to \Delta_{\text{inn}}\to \Delta\to \Delta_{\text{out}}\to 1,
  \]
  where $\Delta_{\text{out}}\subset \Out(G)$ has order prime to the
  characteristic and $\Delta_{\text{inn}}$ is contained in some maximal
  torus of $\Ad(G)$ fixed by $\Delta_{\text{out}}$.  And of course
  $\Delta_{\text{inn}}$ is itself an extension of an \'etale group of order
  prime to the characteristic by its identity component, a multiplicative
  $p$-group.
\end{rems}

\begin{rems}
Note that for any root system arising as the intersection of the root
system of $G$ with a sublattice of the root lattice of $G$, there is a
corresponding diagonalizable subgroup of $G$ such that the centralizer has
that root system (typically with a torus factor as well); this satisfies
the radical condition in any characteristic if the subsystem contains at
least one long root.  (Indeed, $\Lie(G)$ is a highest weight
representation, so is generated as a $G$-module by any highest weight
vector.)  In addition, we should note that if the normalizer of $\Delta$ in
$\Aut(G)$ is disconnected, then $\Delta$ can be twisted by a Galois cocycle
with values in the component group of the normalizer; this in turn induces
a corresponding twist of $H$.  (This corresponds to taking $\Delta$ to lie
in a non-split maximal torus.)
\end{rems}

In light of the Proposition, we say that a type $X$ ``reduces to'' type $Y$
if there is a sequence $X=X_1,\dots,X_m=Y$ such that $X_{i+1}$ is the
simply connected cover of the adjoint of the centralizer of a diagonal
subgroup of $\Aut(X_i)$ satisfying the radical condition.

\begin{prop}
  Any simply connected simple group which is not excluded by the Theorem
  reduces to a group in the following list:
  \begin{itemize}
  \item[(a)] $A_1^{(m)}$ with $p^m>4$
  \item[(b)] $A^{(m)}_2$, ${}^2\!A^{(m)}_2$ with $p^m\in \{3,4\}$
  \item[(c)] $A_4$, ${}^2\!A_4$ with $p=2$.
  \end{itemize}
\end{prop}

\begin{proof}
  For non-simply-laced root systems, we may reduce to a suitable twist of
  the subsystem corresponding to the long roots, giving reductions:
  \begin{align}
  B^{(m)}_n&\to {}^2\!D^{(m)}_{n-1}&C^{(m)}_n&\to A^{(mn)}_1\\
  F^{(m)}_4&\to {}^2\!D^{(m)}_4&G^{(m)}_2&\to A^{(m)}_2.
  \end{align}
  In each case to which the Theorem applies, the reduction is to another
  case to which the Theorem applies, so the Proposition holds in these
  cases.

  For simply-laced cases, we may reduce by removing a single orbit
  of roots:
  \begin{align}
  A^{(m)}_n&\to A^{(m)}_{n-1}&  
  {}^2\!A^{(m)}_{2n}&\to {}^2\!A^{(m)}_{2n-2}&
  {}^2\!A^{(m)}_{2n+1} &\to {}^2\!A^{(2m)}_n\\
  D^{(m)}_n&\to A^{(m)}_{n-1}& {}^2\!D^{(m)}_n &\to {}^2\!D^{(m)}_{n-1}
  &{}^3\!D^{(m)}_4&\to A^{(3m)}_1
  \end{align}
  with $E^{(m)}_8\to E^{(m)}_7\to E^{(m)}_6\to D^{(m)}_5\to A^{(m)}_4$.
  Finally, ${}^2\!D^{(m)}_4$ has a $\mu_2$ with centralizer (corresponding
  to removing the middle node of the affine Dynkin diagram) of type
  $A^{(2m)}_1A^{(2m)}_1$ which may be further twisted (by the affine
  diagram automorphism) to give a reduction to $A^{(4m)}_1$.

  When $p^m>4$, the above reductions take us to $A^{(mm')}_1$ for some
  $m'\ge 1$ as required, while for $p^m\in \{3,4\}$, they take us to
  one of $A^{(m)}_2$, ${}^2\!A^{(m)}_2$, or $A^{(mm')}_1$ for $m'>1$.
  Finally, for $p^m=2$, we either reduce to a group with $p^m>2$
  or to one of $A_4$, ${}^2\!A_4$.
\end{proof}

For type $A^{(m)}_1$ with $p^m>4$, i.e., $\SL_2(R)$ where $R$ is
the corresponding unramified $p$-adic ring, lifting was shown in
\cite[Lem. 3, \S IV.3.4]{SerreJP:1968}, and thus it remains only to verify it in the remaining
six cases.  For the four such cases in characteristic 2, the radical
of the Lie algebra is 0, and thus it suffices to show that any lift
contains {\em some} Lie algebra element.  The same argument as in
the proof of the reduction shows that it suffices to find a centralizer
for which the extension is not split, allowing us to reduce to the
following fact.

\begin{prop}
  The homomorphism $\SL_2(W_2(\F_4))\to \SL_2(\F_4)$ does not split.
\end{prop}

\begin{proof}
  It suffices to check that there is no lift modulo 4 of the $2$-Sylow.  If
  we choose lifts of the generators of the $2$-Sylow:
  \[
  \begin{pmatrix}
    1+2a_1&1+2b_1\\
    2c_1&1+2d_1
  \end{pmatrix},
  \begin{pmatrix}
    1+2a_2&\omega+2b_2\\
    2c_2&1+2d_2
  \end{pmatrix},
  \]
  then the condition is that their lifts also be commuting elements of
  order 2.  The order 2 condition forces $c_1=a_1+d_1+1=c_2=a_2+d_1+1=0$,
  but then the commutator is
  \[
  \begin{pmatrix}
    1 & 2(\omega+1)\\
    0 & 1
  \end{pmatrix}
  \ne 0.
  \]
\end{proof}

\begin{rem}
  Note that this does not imply that the Theorem holds for
  $G(k)=\SL_2(\F_4)$, as we have not ruled out the existence of a lift
  modulo the center.  (Indeed, such a lift exists!)
\end{rem}

It remains to consider the two characteristic 3 cases.  In those cases, the
cosocle of the Lie algebra is the quotient by the center.  Moreover, in
either case, we can find two commuting elements of the $3$-Sylow of the
forms
\[
g_1:=\begin{pmatrix}
  1 & a & b\\ 0&1&c\\ 0&0&1
\end{pmatrix}\!,\  
g_2:=\begin{pmatrix}
  1 & 0 & d\\ 0&1&0\\ 0&0&1
\end{pmatrix}
\]
with $d\ne 0$ and $(a,c)\ne (0,0)$.  The non-diagonal coefficients of the
commutator must still vanish, and vanishing of the $23$ and $32$
coefficients of the commutator forces the $31$ entry of the lift of $g_2$
to vanish.  But then
\[
g_2^3 =  \begin{pmatrix}
  1 & 0 & 3d\\ 0&1&0\\ 0&0&1\end{pmatrix}
  \pmod{9},
\]
and thus we do not get the (undesired) splitting.

This finishes the proof of Theorem \ref{thm:vasiu_fixed}.\qed

\medskip

For completeness, we show that the listed groups are indeed counterexamples
to lifting.  For characteristic 3, we have an injective morphism
\[
U(\H)\to U(\H\otimes\Z_3)\cong \SL_2(\Z_3)
\]
of (reduced) norm 1 subgroups, where $\H$ is the ring of Hurwitz
quaternions, and this surjects on $\SL_2(\F_3)$, so gives the desired
counterexample.

For characteristic 2, we have already noted the examples $W(E_8)\subset
O^+_8(\Z_2)$, $W(E_7)\subset O_7(\Z_2)$, $S_6\subset O_5(\Z_2)$ and
$S_8\subset O^+_6(\Z/4\Z)$, to which we may add the additional examples
$W(E_6)\subset O^-_6(\Z_2)$, $S_5\subset O^-_4(\Z_2)$ and $S_3\subset
O_3(\Z_2)$.  One can also verify by direct calculation that the subgroup
$G_2(\F_2)\subset W(E_7)'$ preserves a suitably general\footnote{I.e., such
that its reduction mod any prime is in the unique dense orbit of $\SL_7$ in
$\wedge^3$.} 3-form over $\Z$, so maps to $G_2(\Z)$.  For $A_2$, we note
that the complex reflection group $ST_{24}$ maps to
$\GL_3(\Z[(-1+\sqrt{-7})/2])$; restricting this map to the derived subgroup
and completing at a prime over 2 gives the desired morphism $\SL_3(\F_2)\to
\SL_3(\Z_2)$.  Finally, for ${}^2\!A_2$, we note that $\GU_3(2)$ is the
normalizer in $\GL_n(\C^3)$ of an extraspecial $3$ group, and thus has a
$3$-dimensional unitary representation over $\Z[\zeta_3,1/3]$, which
restricts to give the requisite unitary representation of $\SU_3(2)$ over
$\Z_2[\zeta_3]$.

\subsection{Unramified Jordan-Pochhammer representations}

For a cyclotomic Jordan-Pochhammer representation at an unramified prime,
it is straightforward to apply the above result to deduce largeness in the
corresponding $p$-adic group.  There is no difficulty for $n>4$, as then
there are no counterexamples $A_n$ or ${}^2\!A_n$.  For $p$ odd, this extends
to $n>2$ (as long as the representation is large mod $p$, that
is!), and even for $p=2$, we can only hit counterexamples when $N\in
\{3,6\}$, the only cases in which the residue field of the cyclotomic field
is $\F_4$.  Of course, we can still prove largeness if we can show by other
means that the degree 1 component of the associated graded Lie algebra has
nonscalar elements.  The simplest way to do so is to exhibit a transvection
in the global group: it has minimal polynomial $g^2-2g+1=0$ and thus $g^2 =
1+2(g-1)$ maps to a nonscalar Lie algebra element.  In particular, if any
set of between $2$ and $n$ parameters multiplies to 1, then lifting holds,
and it is easy to see that the only possible cases in which lifting can
fail (up to Galois) are then $(\zeta_6,\dots,\zeta_6,\zeta_6^2)$ for $n=3$
and $(\zeta_6,\dots,\zeta_6)$ for $n=4$.  In both cases, the global image
is actually finite, and thus lifting does indeed fail to hold.

\subsection{The partially ramified case}

When the cyclotomic field is ramified at $p$, this argument fails, as it
cannot rule out the reductive group over the unramified subfield.  However,
the general principle still applies, at least in the partially ramified
case: for $n>2$, if the degree 1 piece of the associated graded Lie algebra
contains a nonscalar element, then it contains $\sl_n$ or $\su_n$ and thus
the associated graded Lie algebra contains $\sl_n$ or $\su_n$ in every
degree.

In particular, since this involves calculations mod the square of the
prime, we may first quotient by that square.  A choice of generator of the
maximal ideal (say $\pi=\zeta_{p^l}-1$) induces an isomorphism
$R/\frakm^2\cong \F_q[\epsilon]/(\epsilon^2)=:\bar{R}$, and thus we may
rephrase the question in terms of Jordan-Pochhammer representations over
the dual numbers, with parameters of the form $\lambda_i(1+\epsilon
\nu_i)$.  Moreover, since the dual numbers are free of rank 2 over $\F_q$,
this also induces a representation in $\GL_{2n}(\F_q)$.  In particular, any
element of the image has a Jordan decomposition $g=g_s g_u=g_ug_s$ with
$g_s$ semisimple and $g_u$ unipotent, such that $g_s$ and $g_u$ are both
powers of $g$.  (Explicitly, we may take $g_s=g^{q^{(2n)!}}$.)

\begin{prop}
  Suppose $g\in \GL_n(\bar{R})$ reduces to a semisimple element, but is not
  itself semisimple.  Then $\langle g\rangle$ contains a nontrivial Lie
  algebra element.
\end{prop}

\begin{proof}
  The Jordan decomposition respects homomorphisms, and thus $g_u$ reduces
  to the unipotent factor $1$ of $\bar{g}$ as required.
\end{proof}

\begin{rem}
  The simplest sufficient condition for an element to not be semisimple is
  if its characteristic polynomial depends on $\epsilon$.
\end{rem}

\begin{cor}
  Suppose $g\in \GL_n(\bar{R})$ reduces to a semisimple element,
  but is not itself semisimple, and that $\ker(g-1)$ contains a
  nonzero free $\bar{R}$-module.  Then $\langle g\rangle$ contains
  a nonscalar Lie algebra element.
\end{cor}

\begin{proof}
  Here we note that $g_u$ fixes $\ker(g-1)$, and thus the corresponding Lie
  algebra element annihilates $\ker(g-1)\otimes \F_q$; since it is nonzero,
  it cannot be scalar.
\end{proof}

\begin{cor}
  The Jordan-Pochhammer representation over $\bar{R}$ contains a nonscalar
  Lie algebra element unless $(\lambda_0\lambda_i-1)(\nu_0+\nu_i)=0$ for
  all $1\le i\le n+1$.
\end{cor}

\begin{proof}
  If $\lambda_0\lambda_i\ne 1$, then the generator $g_i$ has semisimple
  reduction with $\ker(g_i-1)$ of positive free rank, but cannot itself be
  semisimple unless $\nu_0+\nu_i=0$.  Thus if
  $(\lambda_0\lambda_i-1)(\nu_0+\nu_i)\ne 0$ for some $i$, then some power
  of $g_i$ is the desired nonscalar Lie algebra element.
\end{proof}

There is another useful case in characteristic 2.

\begin{prop}
  Suppose $q$ is a power of 2 and the nonscalar element $g\in
  \GL_n(\bar{R})$ satisfies $(g-(1+\epsilon))(g-1)=0$.  Then $\langle
  g\rangle$ contains a nonscalar Lie algebra element.
\end{prop}

\begin{proof}
  If $\bar{g}=1$, then $g$ is already a nonscalar Lie algebra element.
  Otherwise, $g^2 = 1 + \epsilon (g+1) = 1+\epsilon(\bar{g}-1)$ represents
  a nonzero nilpotent element of the Lie algebra.
\end{proof}

\begin{cor}
  In characteristic 2, the Jordan-Pochhammer representation contains a
  nonscalar Lie algebra element unless $\nu_0=\nu_1=\cdots=\nu_{n+1}$.
\end{cor}

\begin{proof}
  If $\nu_0+\nu_i\ne 0$, either $\lambda_0\lambda_i\ne 1$ and we have
  already exhibited a nonscalar Lie algebra element, or
  $\lambda_0\lambda_i=1$ and $g_i^2$ is the desired nonscalar Lie algebra
  element.
\end{proof}

Since at least one $\nu_i$ must be nonzero (as otherwise the
parameters wouldn't generate the dual numbers), this implies the
existence of a nonscalar Lie algebra element whenever $n$ is odd.
This is already enough to let us prove all characteristic 2 cases
with $n>4$: either $n$ is odd and some $g_S$ works directly, or
there exists some pair $i<j$ with $\lambda_i\lambda_j\ne 1$ (WLOG
$\lambda_n\lambda_{n+1}\ne 1$) and the modular representation with
parameters
\[
(\lambda_0;\lambda_1,\dots,\lambda_{n-1},\lambda_n\lambda_{n+1})
\]
is large and satisfies lifting.  The subgroup generated by
$g_1,\dots,g_{n-1}$ differs only by scalars, so is also large and has
nonscalar Lie algebra elements, and thus we may find an element of that
subgroup that has semisimple reduction with no eigenvalue 1 and nontrivial
unipotent part.  But then the corresponding element in the original
representation still has semisimple reduction and nonscalar unipotent part.

This fails for $n=4$, as there are cases with $n=4$ such that no way of
reducing to $n=3$ has large reduction.  This is not a major issue, however,
and we find that the only such case (up to the obvious symmetries) is
\[
(\zeta_3(1+\epsilon);\zeta_3(1+\epsilon),\zeta_3(1+\epsilon),\zeta_3(1+\epsilon),\zeta_3(1+\epsilon),\zeta_3(1+\epsilon)).
\]
(The situation for $n=2$ is much more subtle.)

We may thus reduce to the case of odd characteristic.

\begin{lem}
  Let $K/\F_q$ be a quadratic \'etale algebra.  Then for $n>2$, every
  transvection in $\SU_n(K/\F_q)$ is conjugate to its inverse.
\end{lem}

\begin{proof}
  We first note that any transvection in $\GU_2(K/\F_q)$ is conjugate to
  its inverse.  Indeed, relative to the anti-Hermitian form
  $\begin{pmatrix} 0&-1\\1&0\end{pmatrix}$, a typical transvection has the
  form
  \[
  \begin{pmatrix}
    1 & \beta\\
    0 & 1
  \end{pmatrix}
  \]
  with $\beta\in \F_q$, and conjugating by an element
  \[
  \begin{pmatrix}
    \alpha & 0\\
    0 & \bar\alpha^{-1}
  \end{pmatrix}
  \]
  multiplies $\beta$ by any norm, so in particular by $-1$.

  For $n>2$, let $g$ be a transvection, and note that $\ker(g-1)$ contains
  a free submodule of rank $n-2$ on which the form is nondegenerate.  The
  stabilizer of such a submodule acts on the orthogonal complement as
  $\GU_2(K/\F_q)$, and thus there is an element $h$ of the stabilizer such
  that $h g h^{-1} g$ fixes the orthogonal complement pointwise.  Since
  both $g$ and $h g h^{-1}$ also fix the submodule pointwise, $h g
  h^{-1}=g^{-1}$ as required.
\end{proof}

In particular, if $\nu_0+\nu_1\ne 0$ but $\lambda_0\lambda_1=1$, then there
is an element $h$ of the image of the Jordan-Pochhammer representation
(assuming that image is $\SU_n$) such that $h g_1 h^{-1}$ has the same
reduction as $g_1^{-1}$.  Both elements are pseudoreflections and thus
their product is an element of the Lie algebra with image of rank at most
2; since its trace is $2(\nu_0+\nu_1)\ne 0$, it gives the desired nonscalar
Lie algebra element.  The one caveat is the requirement that the reduction
be $\SU_n$.  This fails in the totally ramified case, but otherwise works
for all $n>2$: in the two cases that fail to be large, none of the
generators can be transvections!

We thus reduce to the case that $\nu_0+\nu_i=0$ for all $i$.  This is
impossible for $n=4$ (since $4\nu_0=0$), and as in the characteristic 2
case, this produces nonscalar Lie algebra elements for all larger $n$.  For
$p>3$, we can also rule out $n=3$, while for $p=n=3$ there are a couple of
cases for which the above constructions of Lie algebra elements fails.
(Those cases can be dealt with in other ways, but we omit the details.)

\subsection{The totally ramified case}

In the totally ramified case $\lambda_0=\cdots=\lambda_{n+1}=-1$,
(where as in the partially ramified case, these are still the
reductions of the original parameters, which are negatives of
$p^l$-th roots of unity) there are two difficulties.  The first is
that the associated graded Lie algebra is more complicated, so we
need to establish that it is still generated in degree 1 and that
in degree 1 the invariant subspace generated by a nonscalar element
contains the full Lie algebra.  The second is that it is no longer
the case that a transvection is necessarily conjugate to its inverse.

We deal with the latter issue first, and show that there is still a
nonscalar element of the degree 1 piece of the Lie algebra.  We can easily
check that the representation
\[
(g_1,g_2,g_3)
\mapsto
\left(
\begin{pmatrix}
  1 & 2\\
  0 & 1
\end{pmatrix}
,
\begin{pmatrix}
  1 & 0\\
 -2 & 1
\end{pmatrix}
,
\begin{pmatrix}
 -1 & 2\\
 -2 & 3
\end{pmatrix}
\right)
\]
is a Jordan-Pochhammer representation with parameters $(-1,-1,-1,-1)$, and
thus in characteristic $p$, we have
\[
g_1^{(1-p)/2} g_2^{(1-p)/2} g_1^{(1-p)/2}
=
\begin{pmatrix}
  0 & -1\\
  1 & 0
\end{pmatrix}.
\]
It then follows that the element
\[
g_1^{(1-p)/2} g_2^{(1-p)/2} g_1^{(1-p)/2}
\]
is semistable in any Jordan-Pochhammer representation over $\F_p$ with all
parameters $-1$, as it respects a filtration for which the actions on
subquotients have relatively prime minimal polynomials $t-1$ and $t^2+1$.
But then
\[
(g_1^{(1-p)/2} g_2^{(1-p)/2} g_1^{(1-p)/2})^4
\]
is a Lie algebra element, which is nontrivial (having nonzero trace by
virtue of the element of $\GL_2(\bar{R})$ having nontrivial determinant) as
long as $\nu_1\ne \nu_2$.

We can then deal with cases with $n>4$ for which the parameters are all
equal by reducing to the rank $n-2$ case: the group for rank $n-2$ contains
nonscalar Lie algebra elements and multiplying such an element by $-1$
gives an element that still has semisimple reduction in rank $n$ with
nontrivial unipotent part.

For $n=4$, $p>3$ we may use the other relation of $\SL_2(\Z)$
to get a Lie algebra element
\[
(g_1^{(1-p)/2} g_2^{(1-p)/2})^6,
\]
which is nontrivial as long as $2\nu_0+\nu_1+\nu_2\ne 0$.  But for $n=4$,
these equations cannot be simultaneously satisfied for {\em all} pairs, and
thus we have lifting in that case as well.

For $n=4$, $p=3$, lifting in fact fails when $\nu_1=\cdots=\nu_5$ (forcing
$\nu_0=\nu_1$); indeed, this is the reduction of the finite case with
parameters $(\zeta_6,\dots,\zeta_6)$.

\medskip

It remains only to show that the graded Lie algebra is generated in degree
1 and the degree 1 component is generated as a $\Sp_{2n}(\F_p)$-module by
any nonscalar element.  More precisely, since we are including elements in
$\GU$ rather than $\SU$, we are given a nonscalar element of the
corresponding slightly larger Lie algebra (isomorphic as a representation
to $\wedge^2(\F_p^{2n})$, with scalars corresponding to multiples of
$J^{-1}$ where $J$ is the symplectic form), and need to know that the
$\Sp_{2n}(\F_p)$-module it generates contains the true Lie algebra (the
alternating forms with $\Tr(AJ)=0$).  For any form $A$ not proportional to
$J^{-1}$, let $v$ be a vector which is not an eigenvector for $AJ$.  Then
the corresponding transvection adds a form of rank 2 to $A$, and thus we
may reduce to the case that the given nonscalar element has the form
$v_1\wedge v_2$, with $v_1$ and $v_2$ orthogonal with respect to the
symplectic pairing.  But then we can act by $\Sp_{2n}(\F_p)$ to replace
$v_2$ by anything else orthogonal to $v_1$ and then $v_1$ by anything else
orthogonal to the new $v_2$, and thus see that {\em every} such form is
contained in the $\Sp_{2n}(\F_p)$-module, proving the desired result.

Finally, to see that the graded Lie algebra is generated in degree 1, we
note that the graded Lie algebra is periodic with period 2: the degree $k$
space is $\frakg_{k\bmod 2}$ with $\frakg_1$ the space of traceless
matrices $A$ with $AJ$ alternating and $\frakg_0=\sp$ the space of
matrices $A$ with $AJ$ symmetric.  Thus the claim reduces to showing that
$\sp_{2n}=[\frakg_1,\frakg_1]$ and $\frakg_1=[\sp_{2n},\frakg_1]$.  The latter follows by noting that the group is generated by
transvections, and for any transvection, the action $x\mapsto (g-1)x$ on
$\sp_{2n}$ is the same as the action of the Lie algebra element $\lim_{t\to
  0} (g^t-1)/t$.  For the former, since we are in odd characteristic,
$\sp_{2n}$ is an irreducible $\Sp_{2n}$-module, and thus it suffices to
show that $[\frakg_1,\frakg_1]$ is nonzero.  But $\frakg_1$ has
dimension $2n(2n-1)/2-1$, while any abelian subalgebra of $\sl_{2n}$ has
dimension at most $2n-1$, so this is automatic (at least for $2n>2$).
(Note that this fails for $2n=2$ for the simple reason that in that case
the Lie algebra is actually only supported in even degrees, as the quotient
of the group by its center is defined over the real subfield!)

\medskip

Putting everything together, we in particular conclude that lifting holds
for any case with $n>4$, and thus have finished the proof that every
cyclotomic Jordan-Pochhammer representation for $n>4$ is large.

For $n=4$, the only remaining issue is that the isomorphism $\PSp_4(\F_3)\cong
\SU_4(\F_2)$ causes issues with pairwise surjectivity, but this
only arises when $N=6$, $n=4$, and the parameters are all $\pm
\zeta_6^{\pm 1}$.  This is either the usual finite case or contains
a transvection, and the latter has different orders mod $2$ and
$3$, so that we still have pairwise surjectivity.  For $n=3$, the
fact that $\SU_3(\F_2)$ is not perfect causes mild issues, but it
can only arise once and thus again only the finite case fails to
be large.  For $n=2$, the situation is significantly more complicated
(more finite cases, an additional isomorphism $\PSL_2(\F_4)\cong
\PSL_2(\F_5)$, the possibility for the image to be defined over a
subfield), but is likely still tractable.  (In particular, the fact
that the group is generated over its center by two pseudo-reflections
is likely to be significantly useful!)

\section{Pryms of singular curves}

There are two natural ways in which we would like to generalize the result
on largeness of the Jordan-Pochhammer group: we would like to allow more
general base curves, and would like to extend the result to the analogous
monodromy group in characteristic $p$, including the wild case $p|N$.  In
both cases, a key ingredient is the fact that the Prym (and thus the
$L$-torsion-of-Prym sheaf) can have good reduction even when the curves do
not, and in such a case is closely related to a Prym of a
desingularization.  This lets us embed certain Jordan-Pochhammer families
in families over higher genus curves and gives significantly more freedom
in lifting to characteristic 0.  We will thus temporarily digress to
consider how to understand Pryms of (reduced) singular curves.

To start with, let $C$ be a reduced proper curve over a field $k$.  Such a
curve still has a Picard scheme $\Pic(C)$ and thus one may still consider
the identity component $\Pic^0(C)$, classifying line bundles with degree 0
along each component of $C$.  (See the discussion in
\cite[Chap. 9]{BoschS/LutkebohmertW/RaynaudM:1990}.)  If the \'etale cyclic
group $Z$ of order $N$ acts on $C$, we may thus define $\Prym(C,Z)$ to
be the tensor product with the representation $\Z[\zeta_N]$ of $Z$.
Although this behaves badly in families (not even the dimension is constant
when reducing to bad characteristic), it still has useful content.

By a result of Chevalley and Rosenlicht (see \cite[Thm. 1,
  \S9.2]{BoschS/LutkebohmertW/RaynaudM:1990}), any smooth algebraic group
scheme $G$ over a perfect field has a natural filtration $U\subset L\subset
G$ where $U$ is unipotent, $L$ is linear, $L/U$ is reductive, and $G/L$ is
an abelian variety.  (When the base field is not perfect, this filtration
of course exists over some finite inseparable extension!)  When $G$ is
commutative, we thus obtain an induced abelian variety $G/L$ and an induced
{\em semi}abelian variety $G/U$.  In the case of $\Pic^0$, both quotients
are themselves Jacobians of curves.

\begin{lem}\cite[\S9.2]{BoschS/LutkebohmertW/RaynaudM:1990}
  Let $C$ be a reduced proper curve over a perfect field $k$, let
  $\tilde{C}$ be its normalization, and let $C'$ be obtained from
  $\tilde{C}$ by identifying those distinct points lying over the same
  points of $C$ (the ``seminormalization'' of $C$).  Then
  $\Pic^0(\tilde{C})$ is abelian, $\Pic^0(C')$ is semiabelian, and there
  are exact sequences
  \begin{align}
  0\to L\to &\Pic^0(C)\to \Pic^0(\tilde{C})\to 0\\
  0\to U\to &\Pic^0(C)\to \Pic^0(C')\to 0\\
  0\to T\to &\Pic^0(C')\to \Pic^0(C)\to 0
  \end{align}
  where $L$ is a smooth connected linear algebraic group, $U$ is a smooth
  connected unipotent group, and $T$ is a torus.
\end{lem}

Note that although $C'$ is only defined over the perfect closure, the group
$\Pic^0(C'_{k^{\text{perf}}})$ is naturally defined over $k$.  Indeed, the
kernel $U$, being unipotent, is annihilated by some fixed power of $p$,
while multiplication by $p$ on $\Pic^0(C'_{k^{\text{perf}}})$ is
surjective.  Thus for sufficiently large $L$,
$\Pic^0(C'_{k^{\text{perf}}})\cong p^l\Pic^0(C_k)$, and the latter is
defined over $k$.  Since $\Pic^0(C'_{k^{\text{perf}}})$ is semiabelian, its
torus subgroup and thus its abelian quotient
$\Pic^0(\tilde{C}_{k^{\text{perf}}})$ are also defined over $k$.

\begin{cor}
  Let $C$ be a reduced proper curve over a perfect field $k$ with an action
  of the \'etale cyclic group $Z$, which thus acts on the normalization
  $\tilde{C}$ and the seminormalization $C'$.  Then $\Prym(\tilde{C},Z)$ is
  abelian, $\Prym(C',Z)$ is semiabelian, and there are exact sequences
  \begin{align}
  0\to L\to &\Prym(C,Z)\to \Prym(\tilde{C},Z)\to 0\\
  0\to U\to &\Prym(C,Z)\to \Prym(C',Z)\to 0\\
  0\to T\to &\Prym(C',Z)\to \Prym(\tilde{C},Z)\to 0
  \end{align}
  where $L$ is a smooth connected linear algebraic group, $U$ is a smooth
  connected unipotent group, and $T$ is a torus.
\end{cor}

\begin{proof}
  The exact sequences associated to the Jacobian are equivariant, and thus
  there are long exact sequences of tensor products.  Thus each group is an
  image of a group of the desired type (linear, abelian, etc.) so is itself
  a group of that type.
\end{proof}

\begin{rem}
  In fact, the actual kernels of the maps to $\Prym(\tilde{C},Z)$ are
  isogenous to the groups $L\otimes_{\Z[Z]} \Z[\zeta_N]$ and
  $T\otimes_{\Z[Z]}\Z[\zeta_N]$ coming from the algebraic parts of the
  Jacobian, as the next term in the long exact sequence is a subgroup of
  the abelian variety $\Prym(\tilde{C},Z)$, so is proper, and thus has
  finite image in any linear algebraic group.
\end{rem}

\begin{prop}
  Let $C$ be a seminormal curve over a field $k$.  The dimension of the torus
  factor of $\Prym(C,Z)$ is $\varphi(N)(b-n-c+i)$, where $n$ is the
  number of free orbits of $Z$ on the (geometric) branch points of $C$, $b$
  is the number of free orbits of $Z$ on their preimages in $\tilde{C}$,
  $i$ is the number of free orbits on the geometrically connected
  components of $C$, and $i$ is the number of free orbits on the
  geometrically irreducible components of $C$.  In particular, if
  $b-n-c+i=0$, then $\Prym(C,Z)$ is an abelian variety.
\end{prop}

\begin{proof}
  The torus factor $T$ of the Jacobian has a natural resolution of the form
  \[
  1\to T_i\to T_c\times T_n\to T_b\to T\to 1
  \]
  where $T_i$ is the torus associated to the \'etale algebra of
  components of $\tilde{C}$, etc.  The cohomology group schemes of the
  tensor product of this complex with $\Z[\zeta_N]$ agree (up to a shift)
  with the cohomology group schemes of the image under
  $\Tor_1({-},\Z[\zeta_N])$, and are thus annihilated by $N$.  It follows,
  therefore, that the alternating sum of ranks of the tensor product is 0,
  giving the desired result.
\end{proof}

We next want to understand how the Prym behaves when the curve degenerates.
In particular, given a curve over a dvr with smooth generic fiber and
reduced special fiber, how does the Prym of the special fiber relate to the
special fiber of the N\'eron model of the Prym?

Here we note that the Prym construction makes sense for an arbitrary
commutative group scheme (or group algebraic space): if $A$ is a
commutative group scheme with an action of the \'etale cyclic group $Z$ of
order $N$, then $\Prym(A,Z)$ is the tensor product with $\Z[\zeta_N]$.  So
we can generalize this to analogous questions for abelian varieties with
actions of \'etale cyclic groups: how does the Prym of the special fiber of
the N\'eron model relate to the N\'eron model of the Prym?  And how does
the Prym of the special fiber of the N\'eron model relate to the Prym of
the its identity component?

For the effect of taking the identity component, we have the following.

\begin{prop}
  Let $A$ be a commutative group scheme with an action of the \'etale
  cyclic group $Z$.  Then the natural map $\Prym(A^0,Z)\to \Prym(A,Z)^0$
  has \'etale kernel of degree dividing the order of the component group of
  $\Prym(A,Z)$.
\end{prop}

\begin{proof}
  We have a short exact sequence
  \[
  0\to A^0\to A\to A/A^0\to 0
  \]
  which by the Snake Lemma gives an exact sequence of the form
  \[
  K\to \Prym(A^0,Z)\to \Prym(A,Z)\to \Prym(A/A^0,Z)\to 0
  \]
  where $K$ is the kernel of the endomorphism of $A/A^0$ of which
  $\Prym(A/A^0,Z)$ is the cokernel.  Since $A/A^0$ is \'etale, $K$ and
  $\Prym(A/A^0,Z)$ are \'etale group schemes of the same order, so that the
  kernel of $\Prym(A^0,Z)\to \Prym(A,Z)$ is \'etale.  Since $\Prym(A^0,Z)$
  is connected and $\Prym(A/A^0,Z)$ is \'etale, the image of
  $\Prym(A^0,Z)\to \Prym(A,Z)$ is $\Prym(A,Z)^0$ and thus $\Prym(A/A^0,Z)$
  is the component group of $\Prym(A,Z)$.
\end{proof}

For the effect of taking special fibers of N\'eron models, we can simplify
using the following fact.

\begin{lem}
  If $A$ is a commutative group scheme with an action of the \'etale cyclic
  group $Z$ and $Z'\subset Z$ is the maximal subgroup of square-free order,
  then $\Prym(A,Z)\cong \Prym(A,Z')$.
\end{lem}

\begin{proof}
  This reduces to the fact that
  \[
  \Z[\zeta_N]\otimes_Z \Z[Z']\cong \Z[\zeta_{N'}]
  \]
  as $Z'$-modules, where $N'$ is the square-free part of $N$.  This may in
  turn be verified \'etale locally, so reduces to the case $Z\cong \Z/N\Z$.
\end{proof}
  
\begin{lem}
  If $Z_1$, $Z_2$ are \'etale cyclic group schemes of relatively prime
  order and $Z_1\times Z_2$ acts on the commutative group scheme $A$, then
  $\Prym(A,Z_1\times Z_2)\cong \Prym(\Prym(A,Z_1),Z_2)$.
\end{lem}

For abelian varieties, we may define a dual Prym by
\[
\Prym^{\vee}(A,Z):=\Prym(A^\vee,Z)^\vee.
\]
The Prym can be computed as the cokernel of $\chi_Z:=|Z|e_Z$ where $e_Z\in
\Q[Z]$ is a suitable idempotent, and the dual Prym is then the identity
component of the kernel of $\chi_Z$.  But this can be computed instead as
the image of the endomorphism $|Z|-\chi_Z$.  Since that endomorphism
annihilates the image of $\chi_Z$, it factors through the Prym, giving a
factorization
\[
A\to \Prym(A,Z)\to \Prym^\vee(A,Z)\to A
\]
of $|Z|-\chi_Z$.  The composition
\[
\Prym^\vee(A,Z)\to A\to \Prym(A,Z)\to \Prym^\vee(A,Z)
\]
is then multiplication by $|Z|$.  Since the two versions of the Prym have
the same dimension, it follows that they are isogenous, and both isogenies
have kernel annihilated by $|Z|$.

\begin{prop}
  Let $R$ be a dvr with field of fractions $K$ and residue field $k$, and
  let $f_K:A_K\to B_K$ be an isogeny of abelian varieties with
  $n$-torsion kernel.  Then the induced morphism of special fibers of
  N\'eron models has $n$-torsion kernel and cokernel.
\end{prop}

\begin{proof}
  Since $\ker(f_K)\subset A_K[n]$, there is an isogeny $g_K$ such that
  $g_K\circ f_K=[n]$ and $f_K\circ g_K=[n]$.  Since extensions of morphisms
  to N\'eron models are unique, the extensions also satisfy $g\circ f=[n]$
  and $f\circ g=[n]$, and thus so do the actions on special fibers.
  It follows that $\ker(f_k)\subset A_k[n]$ and $\coker(f_k)$ is a quotient
  of $B_k/nB_k$.
\end{proof}

\begin{prop}\cite[Prop. 7.5.3]{BoschS/LutkebohmertW/RaynaudM:1990}
  If the short exact sequence $0\to A'_K\to A_K\to A''_K\to 0$ of abelian
  varieties splits up to an isogeny with $n$-torsion kernel with $n$
  invertible on $R$, then the corresponding sequence
  \[
  0\to A'_k\to A_k\to A''_k
  \]
  of special fibers is exact and $\coker(A_k\to A''_k)$ is $n$-torsion.
\end{prop}

\begin{prop}
  If the short exact sequence $0\to A'_K\to A_K\to A''_K\to 0$ of abelian
  varieties splits up to an isogeny with $n$-torsion kernel, then
  $\ker(A'_k\to A_k)$ and $\ker(A_k\to A''_k)/\im(A'_k\to A_k)$
  are unipotent and $\coker(A_k\to A''_k)$ is $n$-torsion.
\end{prop}

\begin{proof}
  Let $S\subset A'_K$ be an $n$-torsion subgroup such that $A_K/S\cong
  (A'_K/S)\times A''_K$.  (Note that if $B_K\subset A_K$ is an abelian
  subvariety such that $B_K\to A_K\to A''_K$ has $n$-torsion kernel, then
  we may take $S=\ker(B_K\to A''_K)=B_K\cap A'_K$.)  Let $S_p$ be the
  $p$-Sylow subgroup scheme of $S$, where $p$ is the residue
  characteristic.  Then
  \[
  0\to A'_K/S_p\to A_K/S_p\to A''_K\to 0
  \]
  splits up to an isogeny with $n'$-torsion kernel, where $n'$ is the
  $p'$-part of $n$.  It follows from the previous Proposition that the
  complex
  \[
  (A'_K/S_p)_k\to (A_K/S_p)_k\to A''_k
  \]
  of special parts of N\'eron models has cohomology supported on the
  $A''_k$ term, with that term $n'$-torsion.  Since the morphisms
  \[
  A'_k\to (A'_K/S_p)_k\qquad A_k\to (A_K/S_p)_k
  \]
  have $p$-power (and thus unipotent) kernel and cokernel, the claim
  follows.
\end{proof}

\begin{cor}
  Let $R$ be a dvr with field of fractions $K$ and residue field $k$, and
  let $Z$ be an \'etale cyclic group.  If $A_K$ is an abelian variety such
  that $\Prym(A_K,Z)$ has good reduction, then the natural morphism
  $\Prym((A_k)^0,Z)\to \Prym(A_K,Z)_k$ is a surjection with unipotent
  kernel.
\end{cor}

\begin{proof}
  The morphism $A_K\to \Prym(A_K,Z)$ has smooth kernel and is split by a
  $|Z|$-torsion isogeny.  We thus conclude that the corresponding map
  $A_k\to \Prym(A_K,Z)_k$ of special fibers has unipotent kernel and
  $|Z|$-torsion cokernel.  The cokernel is a quotient of the component
  group of $\Prym(A_K,Z)_k$, which is trivial since $\Prym(A_K,Z)$ has good
  reduction, and thus $A_k\to \Prym(A_K,Z)_k$ is a surjection with
  unipotent kernel.  The composition with the relevant endomorphism of
  $A_k$ is 0 (since this holds over $K$) and thus this morphism factors
  through $\Prym(A_k,Z)$, letting us conclude that $\Prym(A_k,Z)\to
  \Prym(A_K,Z)_k$ is also a surjection with unipotent kernel, as is the
  corresponding morphism of identity components.

  In particular, the component group of $\Prym(A_k,Z)$ is unipotent, and
  thus the kernel of the map $\Prym((A_k)^0,Z)\to \Prym(A_k,Z)^0$ is an \'etale
  unipotent group.  It follows that the kernel of the composition
  \[
  \Prym((A_k)^0,Z)\to \Prym(A_K,Z)_k^0\to \Prym(A_K,Z)_k
  \]
  is an extension of unipotent groups, so is itself unipotent.
\end{proof}

We can now prove the main result we will use on good reduction of Pryms.

\begin{prop}
  Let $C/R$ be a projective curve with smooth connected general fiber and
  an action of the \'etale cyclic group $Z$.  If
  $\dim\Prym(\tilde{C}_{k^{\text{perf}}},Z)=\dim\Prym(C_K,Z)$, then
  $\Prym(C_K,Z)$ has good reduction and for all $L$ prime to the residue
  characteristic of $R$, $\Prym(C_k,Z)[L]\cong \Prym(C_K,Z)_k[L]$.
\end{prop}

\begin{proof}
  Since the special fiber of the N\'eron model of $\Pic^0(C_K)$ has
  identity component $\Pic^0(C_k)$, we have a natural morphism
  \[
  \Prym(C_k,Z)\cong \Prym(\Pic^0(C_K)_k^0,Z)\to \Prym(\Pic^0(C_K),Z)_k\cong
  \Prym(C_K,Z)_k,
  \]
  and thus (temporarily replacing $k$ by its perfect closure) a composition
  to the quotient $A$ of $\Prym(C_K,Z)_k$ by its maximal smooth linear
  algebraic subgroup.  The corresponding quotient of $\Prym(C_k,Z)$ is
  $\Prym(\tilde{C}_{k^{\text{perf}}},Z)$, so that we have a natural
  morphism $\Prym(\tilde{C}_{k^{\text{perf}}},Z)\to A$ with linear
  algebraic kernel.  But $\Prym(\tilde{C}_{k^{\text{perf}}},Z)$ has no
  positive-dimensional linear algebraic subgroup and thus
  \[
  \dim\Prym(\tilde{C}_{k^{\text{perf}}},Z)\le \dim(A)\le
  \dim\Prym(C_K,Z)_k=\dim\Prym(C_K,Z).
  \]
  We thus conclude that $A\cong \Prym(C_K,Z)_k$ and thus that
  $\Prym(C_K,Z)$ has good reduction.

  It follows that $\Prym(C_k,Z)\to \Prym(C_K,Z)_k$ is a surjection with
  unipotent kernel, and thus induces an isomorphism of $L$-torsion for all
  $L$ prime to the characteristic.
\end{proof}

\section{Ramification data}

Similarly, we need to understand how ramified covers behave when the base
curve degenerates.

If $C/S$ is a (reduced, but possibly singular) curve, and $G/S$ is an
\'etale group scheme, then a $G$-cover of $C$ ramified along a divisor $Z$
is \'etale on the complement of $Z$, so determines a class in in
$H^1(C\setminus Z;G)$.  (All cohomology will be \'etale unless otherwise
specified.)  When $G$ is abelian (in particular in the cyclic case of
present interest), this fits into a long exact sequence
\begin{align}
0&\to H^0(C;G)
\to H^0(C\setminus Z;G)
\to H^1_Z(C;G)\notag\\
&\to H^1(C;G)
\to H^1(C\setminus Z;G)
\to H^2_Z(C;G)
\to H^2(C;G)\to\cdots
\end{align}
This suggests that we should view the local cohomology group
$H^2_Z(C;G)$ as classifying the local structure of the ramification
(``ramification data'').  Given such a class, there is then an obstruction
in $H^2(C;G)$ to the existence of a global ramified cover with the
given ramification data, and when such a cover exists, we may twist by an
\'etale cover, with $H^1_Z(C;G)$ mapping to \'etale covers of $C$ that
become trivial after removing $Z$.

We should caution here that while removing $Z$ does not lose any
information when $C$ is smooth, removing a singular point {\em can} lose
information about the cover.  However, there does not appear to be a way to
retain that additional information without breaking the group structure on
the family of covers.

It turns out that the primary issue when deforming a cover is to deform the
ramification data, as opposed to the cover itself.

\begin{prop}\label{prop:covers_lift}
  Let $R$ be a Henselian dvr, let $C/R$ be a flat proper curve with
  geometrically connected special fiber, and let $Z\subset C$ be any closed
  subset.  Let $\hat{C}_k\to C_k\setminus Z_k$ be a Galois cover with group
  $G$, and let $\phi\in H^2_Z(C;G)$ be a class restricting to the class
  corresponding to $\hat{C}_k$.  Then there is a $G$-cover of $C\setminus Z$
  with special fiber $\hat{C}_k$ and ramification data $\phi$.
\end{prop}

\begin{proof}
  Since $\phi$ restricts to the ramification data of an actual cover, the
  image of the corresponding obstruction under $H^2(C_R;G)\to H^2(C_k;G)$
  is trivial.  Since $C$ is proper with geometrically connected special
  fiber, the reduction map $H^*(C_R;G)\to H^*(C_k;G)$ is a
  quasi-isomorphism, and thus the obstruction is trivial.  There is thus
  {\em some} cover of $C_R$ with ramification data $\phi$.  This may have
  the wrong special fiber, but will differ by a class in the image of
  $H^1(C_k;G)\cong H^1(C_R;G)$, and thus we may twist by a global \'etale
  cover to make the special fiber correct.
\end{proof}

By excision, when computing $H^p_Z(C;G)$, we may replace $C$ by any
\'etale neighborhood of $Z$, or by the limit over all such neighborhoods.
In particular, we may decompose it into contributions from each connected
component of $Z$ (or, \'etale locally, from each geometrically connected
component of $Z$).  

\begin{lem}
  Let $C/S$ be a smooth curve over a scheme $S$ on which $N$ is invertible,
  and let $z:S\to C$ be a section.  Then $H^2_z(C;\mu_N)\cong
  \Gamma(S;\Z/N\Z)$.
\end{lem}

\begin{proof}
  This follows from cohomological purity: since $C$ and the image of $z$
  are both smooth over $R$, with the latter of codimension 1,
  $Rz^!\mu_N\cong (\Z/N\Z)[-2]$, and thus $H^2_z(C;\mu_N)\cong
  \Gamma(S;\Z/N\Z)$ as required.
\end{proof}

If $C$ is proper, there is a homomorphism $H^2(C;\mu_N)\to
\Gamma(S;\Z/N\Z)$ corresponding to the obstruction that the total weight of
a $\mu_N$-cover of a smooth proper curve is $0$.  When this obstruction
vanishes, there is a further obstruction in $H^1(S;\Pic_S(C)[N])$: the
ramified covers form a torsor over the group of \'etale covers, which maps
to $\Pic_S(C)[N]$.  The vanishing of this obstruction corresponds to the
existence of a point in the Picard scheme with specified image under
multiplication by $N$, and thus there is a remaining Brauer obstruction in
$H^2(S;\mu_N)$.  (This filtration of the group of obstructions corresponds
to that coming from the obvious spectral sequence.)

\medskip

For one of our arguments below, we will also need to consider some cases
with ramification in codimension 2.  The following is essentially an
observation of \cite{ArtinM:1977}, corresponding to the fact that rational
double points are quotient singularities.

\begin{prop}\label{prop:smoothing_AN}
  Let $X$ be a surface over a field $k$, $N$ a positive integer invertible
  on $k$, and let $x\in X$ be a rational double point of split type
  $A_{mN-1}$.  Then there is a natural morphism $\Z/N\Z\to H^2_x(X;\mu_N)$.
\end{prop}

\begin{proof}
  Note first that $X$ has a cover ramified at $x$, namely a surface with
  split $A_{m-1}$-singularity.  Indeed, a surface with split
  $A_{m-1}$-singularity has the form
  \[
  \Spec(k[x,y,t]/(xy-t^m)),
  \]
  while the quotient by the action of $\mu_N$ such that $x$ has weight $1$,
  $y$ has weight $-1$ and $t$ has weight 0 is
  \[
  \Spec(k[X,Y,t]/(XY-t^{mN})),
  \]
  where $X=x^N$, $Y=y^N$.  This cover is not \'etale at the origin, and
  thus the corresponding class in $H^1(X\setminus \{x\};\mu_N)$ has
  nontrivial image in $H^2_x(X;\mu_N)$.  Since the same applies to all
  intermediate covers, the image in $H^2_x(X;\mu_N)$ has order $N$ as
  required.
\end{proof}

\begin{rem}
  Note that it would have sufficed to consider the case $m=1$, as all of
  the other cases are base changes of that case.  This would not be a
  particular simplification here, but might be of future use when dealing
  with more complicated singularities (e.g., $\Z/p\Z$-covers of cusps).
\end{rem}

We may think of this as a family of covers depending on $t$; for $t\ne 0$,
it is \'etale, while for $t=0$ it is a ramified cover of the nodal curve
$xy=0$, which on the normalization corresponds to a cover with weights $\pm
1$.  (Indeed, every family of nodal curves is \'etale-locally of this form
for some $m$.)  In particular, this corresponds to the fact that
``admissible'' covers are smoothable, but lets us obtain precise control
over the ramification and residue field extension.

If $\hat{C}_k\to C_k$ is a $\mu_N$-cover in characteristic prime to $N$
with $C_k$ smooth, and $x$ and $y$ are two ramification points with
opposite weights, then the curve $C^+_k$ obtained by gluing $x$ and $y$ is
smoothable to a curve $C^+_R$ over a Henselian dvr with residue field $k$.
Moreover, at the cost of replacing $R$ by a totally ramified extension of
degree dividing $N$, we may arrange for $C^+_R$ to have an $A_{mN-1}$
singularity at the node of $C^+_k$.  We may extend the ramification data at
the other ramification points to $C^+_R$ (extend them to sections and
assign the same weight) and assign the appropriate power of the
ramification data corresponding to Proposition \ref{prop:smoothing_AN} at
the node.

To apply Proposition \ref{prop:covers_lift}, we would need to
show that the resulting ramification data is unobstructed on the special
fiber.  This is not true in general, but something slightly weaker holds:
the image of the obstruction in $H^2(C^+_k;\mu_N)$ vanishes.  The issue is
that the ramification data on $C^+_k$ has the correct weights, but there is
an additional arithmetic contribution to the local cohomology at the node.
However, having the correct weights ensures that the induced ramification
data on $C_k$ agrees with the original ramification data, so is
unobstructed, and thus the image in $H^2(C_k;\mu_N)$ under either
composition in
\[
\begin{CD}
  H^2_{Z\cup x}(C^+_k;\mu_N) @>>> H^2_{Z\cup x\cup y}(C_k;\mu_N)\\
  @VVV @VVV\\
  H^2(C^+_k;\mu_N) @>>> H^2(C_k;\mu_N)
\end{CD}
\]
vanishes.  Applying the long exact sequence
\[
\cdots
\to H^1(C_k;\mu_N)
\to H^1(k;\mu_N)
\to H^2(C^+_k;\mu_N)
\to H^2(C_k;\mu_N)
\to \cdots,
\]
we see that there may still be a residual obstruction in $H^1(k;\mu_N)\cong
H^1(R;\mu_N)$.  It follows that if $R^+$ is the splitting field of the
corresponding $\mu_N$-cover, then the obstruction in $H^2(C^+_{R^+};\mu_N)$
vanishes, and thus Proposition \ref{prop:covers_lift} applies to
give a global $\mu_N$-cover of $C^+_{R^+}$ as prescribed.

Note here that the obstruction only depends on the original cover (and the
choice of points to glue), and not on the choice of smoothing, so there is
a fixed $\mu_N$-cover $k^+$ of $k$ such that for any smoothing of
$C^+_{k^+}$ and any extension of the remaining ramification data, there is
a corresponding extension of the cover.

In terms of the surface $k[x,y,t]/(xy-t^{nm})$ (viewed as a family over
$k[t]$) what is going on is that if we consider the $d$-th power of the
class of Proposition \ref{prop:smoothing_AN} (assuming WLOG that $d|N$),
then the induced ramified cover of the complement of the node on the
special fiber has the form
\[
z^N=x^d,\qquad w^N=y^d,
\]
while the general cover with those weights has two parameters $\zeta$,
$\omega$, and looks like
\[
z^N=\zeta x^d,\qquad w^N=\omega^{-1} y^d.
\]
(The inverse on $\omega$ reflects the fact that we should really be taking
a root of $y^{-d}$.)  To give the corresponding ramification data on the
special fiber, we must quotient by unramified extensions of the special
fiber, which translates to dividing both $\zeta$ and $\omega$ by the same
quantity.  Since this leaves $\zeta/\omega$ alone, we need to choose an
$N$-th root of $\zeta/\omega$ in order to make the ramification data agree.

Note that when the node is not totally ramified, this choice of $N$-th root
can affect the cover of the singular fiber, as it changes the
identification between the two \'etale schemes of preimages.  Since this
acts transitively on the set of equivariant identifications, we see that we
can still use the $H^1(C;R)$ freedom to ensure that we obtain the correct
covering on the special fiber.  (We in particular see that changing the
identification comes from an action of $H^1_x(C_k;G)$, which is not
faithful in general.)

\medskip

For $\Z/p^l\Z$-covers in characteristic $p$, the situation is much more
complicated, as the groups $H^2_z(C;\Z/p^l\Z)$ are infinite-dimensional.
However, since $(C,z)$ is \'etale-locally isomorphic to $(\P^1,0)$, one can
at least identify the ramification data at each point with an element of
the group $H^1(\P^1\setminus \{0\};\Z/p^l\Z)$ classifying $\Z/p^l\Z$-covers
of $\P^1$ ramified only at $0$.  (This is an observation of
\cite{KatzNM:1986}.)  Of course, this identification is not canonical, as
it depends on a choice of uniformizer at $z$.  For any class in
$H^1(\P^1\setminus \{0\};\Z/p^l\Z)$, the subgroup spanned by its images
under changes in uniformizer is finite-dimensional, and thus we obtain a
well-behaved variety for each ``analytic'' isomorphism class of
ramification data.  In particular, given a family of curves with marked
points, there is a covering family in which each marked point is assigned
ramification data from a specific analytic isomorphism class.  (One could
also consider a larger family in which one only specifies the contributions
of each ramification point to the genera of the intermediate covers; it is
unclear to what extent that larger family is connected.)

Given a curve $C_R$ over a Henselian dvr of characteristic $p$ and a
$\Z/p^l\Z$-cover of $C_k$ ramified at rational points of the smooth locus,
we can extend the ramification points to sections of $C_R$, then choose
global uniformizers on each section.  This lets us identify each group
$H^2_Z(C;\Z/p^l\Z)$ with $H^1(\P^1_R\setminus \{0\};\Z/p\Z)$ and thus give
a corresponding constant extension of the ramification data.  By
Proposition \ref{prop:covers_lift}, there is then a cover of
$C_R$ sgreeing with the original cover of $C_k$ and having the ``same''
ramification data on the marked sections.  This extends in the obvious way
to $\Z/p^l\Z\times \mu_{N'}$-covers with $\gcd(p,N')=1$, as the additional
ramification data is an assignment of an element of $\Z/N'\Z$ to each
marked section, which extends canonically over $R$.

For mixed characteristic, the situation is rather more complicated, as each
wild ramification point of the special fiber is a coalescence of multiple
ramification points on the general fiber.  It is a nontrivial fact
\cite{PopF:2014} that ramification data for {\em cyclic} groups lift to
characteristic 0.  We only note the following weak form (of the ``Oort
conjecture''), which suffices for our purposes.

\begin{prop}\cite{PopF:2014}\label{prop:oort}
  Let $k$ be an algebraically closed field of characteristic $p$ and
  $\hat{C}_k\to C_k$ be a ramified $\Z/N\Z$-cover with $\hat{C}_k$ smooth.
  Then there is a mixed characteristic dvr $R$ with residue field $k$ and a
  smooth curve $\hat{C}_R$ with special fiber $\hat{C}_k$ and an extension
  of the action of $\Z/N\Z$.
\end{prop}

\begin{rem}
  The result of \cite{PopF:2014} allows $k$ not to be algebraically closed and
  gives a lift over a finite extension $R$ of the ring $W(k)$ of Witt
  vectors, which can even be given as the compositum with an algebraic
  extension of $\Z_p$.  For our application, however, we need the
  ramification points of the generic fiber to be defined over the function
  field, which requires a further extension in general.  We can thus only
  apply lifting in scenarios where we do not need to control the extension
  of residue fields, so may as well assume they are algebraically closed.
\end{rem}

With this in mind, we define the {\em multiplicity} of a ramification point
in bad characteristic to be the number of points over it in a lift to
characteristic 0.  This is 1 for tamely ramified points, while for wildly
ramified points with wild ramification data given by a $\Z/p^l\Z$-cover
$C\to \P^1$ ramified only at 0, the multiplicity is
\[
\frac{g(C)-g(C/(\Z/p\Z))}{p^l-p^{l-1}}
\]
(i.e., the rank of $\Prym(C,\Z/p^l\Z)$ as a $\Z[\zeta_{p^l}]$-module).

\medskip

We also note that the node-smoothing construction considered above applies
in the bad characteristic setting, as long as the two points being glued
are tamely ramified.  Indeed, if $N=p^lN'$ with $\gcd(p,N')$, then we can
separately extend the $\mu_{N'}$-cover (at the cost of a $\mu_{N'}$-cover of the
residue field) and the $\Z/p^l\Z$-cover.  Here the issue is that there is
an obstruction to the existence of an \'etale cover of the nodal curve that
induces the given \'etale cover of the smooth curve.  The relevant piece of
the long exact sequence is
\[
\to H^1(C;\Z/p^l\Z)
\to H^1(\tilde{C};\Z/p^l\Z)
\to H^1(k;\Z/p^l\Z)\to
\]
so that again the obstruction is killed by some fixed \'etale cover of
$R$.  We thus see that we can extend the $\mu_{N'}\times \Z/p^l\Z$-cover at
the cost of making a $\mu_{N'}\times \Z/p^l\Z$-extension of the residue
field.

\section{Extending to higher genus base curves}

Now that we have shown that (in characteristic 0) the monodromy of Pryms is
large in the case of cyclic $N$-gonal curves with at least 7 ramification
points, the next natural question is whether this extends to higher genus
base curves.  Of course, once the base curve has higher genus, it has
moduli, which leads to some additional variations on the family.  Note also
that the $\Z[\zeta_N]$-rank of the Prym has an additional contribution of
$2g$, and thus to have rank $n$, one should have $n+2-2g$ ramification
points.

The most natural analogue of the family we considered in genus $0$ again
involves varying a single point.  The one caveat is that, as we have seen,
the ramification data does not suffice to determine a cover.  There are
actually two issues: one is that there are multiple geometric isomorphism
classes of covers for a given collection of ramification data, while the
other is that a given geometric isomorphism class may fail to be
represented uniquely by a rational isomorphism class, due to the presence
of cyclic twists.

These issues can be combined into a single observation: for a given family
of ramification data, the moduli of compatible $\mu_N$-covers is a torsor
over the $N$-torsion in the Picard {\em stack} of the base curve.  To be
precise, given ramification data $m_0,\dots,m_{n+1-2g}\in \Z/N\Z$ at points
$x_0,\dots,x_{n+1-2g}\in C$, an actual cover requires specifying a line
bundle ${\cal L}$ and an isomorphism ${\cal L}^{\otimes N}\cong {\cal
  L}(\sum_i m_i[x_i])$ (for lifts of $m_0,\dots,m_{n+1}$ to $\Z$, the
choices of which do not affect the family).  This of course is only
determined up to isomorphism, but is still a gerbe over the preimage under
$[N]$ of the line bundle with given divisor.  (That this is a gerbe should
not be a surprise: after all, the objects we are classifying have
nontrivial automorphisms, namely the $\mu_N$ action!)

There are two natural ways to deal with this gerbe structure.  The first is
to fix some unramified point of $C$ (analogous to $\infty$ in the genus 0
family) and mark one of its preimages in the $\mu_N$-cover.  Since the
fixed automorphism group acts strictly transitively on the preimages, the
corresponding moduli space is precisely the na\"{i}ve one, i.e., the torsor
over $\Pic^0(C)[N]$.  (There would normally be a Brauer obstruction to the
existence of a line bundle in a given rational geometric isomorphism class,
but marking a point also kills that obstruction.)  The one caveat is that
the requirement that the point be unramified introduces additional boundary
components to the moduli space.  (Of course, the monodromy around such
components simply changes the marked point, and thus the action on the Prym
lies in $\mu_N$.)

The other approach is somewhat more subtle, and involves the observation
that the ``degerbification'' can itself be viewed as a moduli problem.
More precisely, we may define a ``pseudocover'' to be the \'etale
sheafification of the notion of a $\mu_N$-cover modulo cyclic twists.  In
particular, a $\mu_N$-pseudocover is determined by a point ${\cal L}\in
\Pic(C)$ such that $N{\cal L}\cong [D]$, where $D$ is given by the
ramification data.  There is then an \'etale covering $U/S$ over which this
point in $\Pic(C)$ is represented by a line bundle, and a choice of
isomorphism then gives rise to an actual $\mu_N$-cover, with different
choices corresponding to different cyclic twists.  The cost to this
approach is that we cannot define the Prym of a pseudocover.  However, if
the pseudocover is split by the \'etale cover $U$, then the Prym, and thus
its $L$-torsion, can be defined over $U$.  Moreover, although
$\Prym(\hat{C})[L]$ itself does not descend to $U$, the quotient sheaf of
sets $\Prym(\hat{C})[L]/\mu_N$ {\em does}.  The action of an element of the
fundamental group on the quotient sheaf determines a {\em projective}
action on the original sheaf (i.e., an action up to scaling by $\mu_N$),
and thus one may still consider largeness in this setting.

The one advantage of the pseudocover approach is that there is never any
difficulty in defining it, so in particular given a family of ramification
data over a proper base, the corresponding family of pseudocovers will
itself be proper.  (This includes points where the covering curve itself
either fails to be flat or becomes singular.)  Moreover, pseudocovers
naturally form a group.

With this in mind, we turn to the question of the analogue of our genus 0
family, i.e., in which all but one ramification point is fixed.  With this
in mind, let us fix an irreducible $\mu_N$-cover $\hat{C}/C$ and a
ramification point $x_0\in C$.  This induces a family of ramification data
as we vary $x_0$ and thus a corresponding family of pseudocovers, which
will be a $\Pic(C)[N]$-torsor over $C$.  This, of course, leads to an
immediate question of whether this torsor is connected, and if not what its
components are.  In fact, it can only be connected if $x_0$ is totally
ramified: if $x_0$ has ramification degree $d$, then the change in the
intermediate $\mu_{N/d}$-cover is \'etale, and thus must be locally
constant as we vary $x_0$.  We should thus restrict to the component(s) for
which this intermediate cover is trivial.  We thus are left with
considering the $\Pic(C)[d]$-torsor over $C$ that classifies
$\mu_d$-pseudocovers with precisely two totally ramified points, one of
which is $x_0$.

We thus need to understand this torsor better.  It turns out that this is
closely related to a certain universal $\Pic(C)[d]$-torsor over $C$.  There
is a natural map $C\to \Pic^1(C)$ and thus we may consider the composition
$C\times \Pic^1(C)\to \Pic^1(C)^2\to \Pic^0(C)$ where the second map is
$(x,y)\mapsto yx^{-1}$.  This is clearly $\Pic^0(C)$-equivariant relative
to the natural action of $\Pic^0(C)$ on $\Pic^1(C)$, and thus is also
$[d]\Pic^0(C)$-equivariant, so that there is an induced morphism of stacks
\[
C\times [\Pic^1(C)/[d]\Pic^0(C)]\to [\Pic^0(C)/[d]\Pic^0(C)].
\]
Since $[d]$ is surjective, the latter quotient stack is nothing other than
the classifying stak $B\Pic^0(C)[d]$, and thus one may view this instead as
a morphism
\[
C\times [\Pic^1(C)/[d]\Pic^0(C)]\to B\Pic^0(C)[d].
\]
Taking the fiber product with $S\to B\Pic^0(C)[d]$ gives a
universal family of $\Pic^0(C)[d]$-torsors over $C$ parametrized by
$[\Pic^1(C)/[d]\Pic^0(C)]$.  Note that since the latter has a single
geometric point, the torsors in this family are all geometrically
isomorphic.

\begin{lem}
  This covering curve is irreducible.
\end{lem}

\begin{proof}
  Since the torsors are geometrically isomorphic, we may reduce to
  considering the pullback of $f:C\to \Pic^0(C)$ by $[d]$ where $f(x)=0$
  for some point $x\in C$.  This is an \'etale cover of a smooth projective
  curve, so irreducibility is inherited from generalizations, and thus we
  may lift to characteristic 0 and work over $\C$, where it reduces to
  surjectivity of $\pi_1(C,x)\to \pi_1(\Pic^0(C),0)[d]\cong H_1(C;\Z)[d]$.
\end{proof}

Pulling back through the morphism $C\to [\Pic^1(C)/[d]\Pic^0(C)]$ gives a
torsor over $C\times C$, classifying triples $(x,y,{\cal L})$ with $x,y\in
C$ and ${\cal L}^{\otimes d}\cong {\cal L}(y-x)$, the moduli space of
$\mu_d$-pseudocovers with two ramification points of exponents $\pm 1$.
(The pseudocover of course degenerates to the trivial cover when the points
coincide.)

Since we can always multiply the weights in our ramification data by units
in $\Z/N\Z$, we may assume that our original ramified point $x_0$ has
weight $N/d$, and thus that the torsor we needed to understand is the fiber
over $x=x_0$ of the above family of torsors, and is thus an irreducible
$\Pic(C)[d]$-torsor over $C$.  We can then either consider the quotient
sheaf $\Pic(\hat{C})/\mu_d$ on the resulting family of pseudocovers
(which is \'etale on the complement of the preimages of the fixed
ramification points) or can mark a point which is unramified for the
original cover and consider the corresponding sheaf $\Pic(\hat{C})$ (which
is \'etale on the complement of the preimages of the fixed ramification
points and the fixed unramified point).  In either case, we may ask whether
the monodromy is large, and the two answers will be the same.

Similar considerations apply if we allow all of the ramification points to
move; in that case, the only difference is that we should replace $d$ by
the least common multiple of the ramification degrees.  (If that lcm is not
$N$, then we may as well fix the intermediate \'etale $\mu_{N/d}$-cover,
which we should assume irreducible.)  This gives a sheaf on a
$\Pic(C)[d]$-torsor over $C^{n+2-2g}$ (or over $C^{n+3-2g}$ if we mark an
additional unramified point), and again we may ask whether the monodromy is
large.

Finally, we may allow $C$ itself to vary.  If $d=N$, then the corresponding
family is a $\Pic(C)[N]$-torsor over ${\cal M}_{g,n+2-2g}$ (for pseudocovers)
or ${\cal M}_{g,n+3-2g}$ (for the version with additional marked point), while
otherwise we get a $\Pic(C)[d]$-torsor over the fiber product with the
family of irreducible \'etale $\mu_{N/d}$-covers of genus $g$ curves.
(This is contained in the relative $\Pic(C)[N/d]$ as the component where
the point has exact order $N/d$; that this is irreducible is well-known.)

Of course, allowing additional things to vary can only make the monodromy
larger, and thus we will aim whenever possible to work with the
single-varying-point family.  (This also has the advantage of being
$1$-dimensional, and thus making descent to finite characteristic somewhat
simpler.)  The main exception (apart from some issues with wild
ramification) is that the single-varying-point family requires there to
{\em be} a ramification point to vary, and thus we must allow the curve to
vary if we want to consider \'etale covers.

\medskip

The idea for extending to higher genus is that if an \'etale sheaf extends
along a dvr, then the monodromy of the general fiber contains the monodromy
of the special fiber.  In particular, in our context, this implies that if
$\hat{C}/R$ is a $\mu_N$-pseudocover with smooth general fiber and Prym of
good reduction (this is independent of cyclic twist, so still makes
sense!), then the monodromy action on $\Prym(C_k)/\mu_N$ lifts to $C_K$ and
thus to the varying curve family.  Since $C_k$ is allowed to be singular,
its Prym is built out of lower-dimensional Pryms, so allows us to induct.
Most of the time, of course, this only gives a relatively small subgroup of
monodromy (though see \cite{AchterJD/PriesR:2007} for a way to use
compatible pairs of degenerations to gain traction), but it turns out that
there is a way to degenerate so that the result still has large monodromy.

We consider, in particular, the boundary component $\Delta_0$ on which the
curve remains irreducible but acquires a node.  As we have seen, the node
may be ramified on the degenerate cover; when this happens, there is no
free orbit of nodes, and thus no torus contribution to the Prym.  It
follows that the Prym has good reduction along any dvr degenerating to such
a cover, and thus that the $L$-torsion-of-Prym sheaf extends.  Working in
reverse, let $\hat{C}_{g-1}\to C_{g-1}$ be a $\mu_N$-cover of a smooth
curve of genus $g-1$ with $n+4-2g$ marked ramification points, such that
two of the ramification points have opposite weight.  Then as discussed
following Proposition \ref{prop:smoothing_AN}, the nodal curve obtained by
gluing the two ramification points has an equivariant smoothing over some
dvr with residue field a $\mu_N$-extension of the field of definition of
$C_{g-1}$.

Now, consider a family of such genus $g-1$ covers with large monodromy,
parametrized by a stack ${\cal X}/k$ with trivial generic stabilizer, and
let $\hat{C}_{g-1}\to C_{g-1}$ be the generic cover in the family, defined
over $\bar{k}({\cal X})$.  The action of the fundamental group on an
\'etale sheaf has the same image as the action of $\Gal(\bar{k}({\cal
  X}))$, so the latter also has large image.  Let $l$ be the
$\mu_N$-extension of $\bar{k}({\cal X})$ discussed following Proposition
\ref{prop:smoothing_AN}, and let $C^+_R$ be any smoothing of the nodal
curve to a Henselian dvr with residue field $l$ and field of fractions $L$.
Then the cover extends to $C^+_R$ and the image of the action of $\Gal(L)$
on the resulting Prym contains the image of the action of $\Gal(l)$.  But
$\Gal(l)$ is the kernel of a morphism $\Gal(\bar{k}({\cal X}))\to \mu_N$,
and thus (since our rank assumption ensures that $\SU$ is perfect) is still
large!

For instance, if we start with a fixed cover over an algebraically closed
field with a chosen ramification point not being glued, then we may first
choose an equivariant smoothing over the original field and then base
change to the field of definition of the corresponding single-varying-point
family.  This modified cover extends (over a $\mu_N$-cover of the field of
definition) as well, and its generic fiber is a base change of a
single-varying-point family, so can only have smaller monodromy than that
family.  We thus conclude inductively (and using the fact that in
characteristic 0, the monodromy is constant in families) that
single-varying-point families always have large monodromy.

\begin{thm}
  Fix $N>1$, $n\ge \max(5,2g-1)$, and let $\hat{C}\to C$ be a
  $\mu_N$-cover of a smooth genus $g$ curve of characteristic 0 with
  $n+2-2g$ ramification points.  Let $x_1$ be one of those points, with
  ramification degree $d$.  Then the monodromy of the family of covers
  obtained by varying $x_1$ is large.
\end{thm}

\begin{rem}
  One technical issue is that if we are not careful about the ramification
  degrees on the points being glued, the monodromy of the genus $g-1$
  family might acquire additional reducible primes.  Of course, we can
  avoid this possibility by only gluing totally ramified points.
\end{rem}

\begin{rem}
  We can replace $5$ by $4$ above by noting that the counterexamples we
  encountered for $n=4$ had no pairs of exponents adding to 0.
\end{rem}

This implies all of the other cases of interest other than that of
unramified covers, for which the only difference is that since one of the
points being glued must be movable, the smoothing cannot descend, and thus
we do not get a constant-curve family.

\begin{thm}
  For any $N>1$ and $g>3$, the monodromy of the Prym of a general
  irreducible \'etale $\mu_N$-cover of a smooth genus $g$ curve of
  characteristic 0 is large.
\end{thm}

\begin{rem}
  Similarly, this can be extended to $g\ge 3$ without too much difficulty.
\end{rem}

\section{Reducing to curves of finite characteristic}

\subsection{Monodromy in good characteristic}

Although it is already of interest to know that the monodromies of Pryms
are large in characteristic 0, the most interesting applications require
understanding the monodromy in finite characteristic.  As one expects, the
complexity of this depends greatly on whether the characteristic divides
$N$.  Note that in any case, we must modify the definition of ``large'' in
characteristic $p$ to exclude the factors corresponding to $p$-power
torsion.  (The $p^l$-torsion of the Prym is badly ramified at $p$, so
although there is still a reasonable notion of monodromy, it is
inaccessible from characteristic 0.)

Since it is no longer a priori the case that the single-varying-point and
fixed-curve families have monodromy indepdendent of the fixed input data,
we focus first on the single-varying-point case.  This not only gives the
strongest result (at least for {\em ramified} covers!), but simplifies the
argument, as it lets us descend the monodromy from a smooth curve rather
than a higher dimensional scheme or even stack.

\begin{thm}
  Fix $N>1$, $n\ge \max(5,2g-1)$, and let $\hat{C}\to C$ be a
  geometrically connected $\mu_N$-cover of a smooth genus $g$ curve of
  characteristic $p\nmid N$ with $n+2-2g$ ramification points.  Let
  $x_1$ be one of those points, with ramification degree $d$.  Then the
  monodromy of the family of covers obtained by varying $x_1$ is large.
\end{thm}

\begin{proof}
  Choose a lift of the tuple $(C,x_1,\dots,x_{n+2-2g})$ over a mixed
  characteristic dvr $R$.  Then the family of pseudocovers obtained by
  varying $x_1$ is a punctured version of the natural $\Pic(C)[d]$-torsor
  over $C$.  The ramification around the preimages of $x_i$ are
  pseudo-reflections with eigenvalue $\zeta_N^{m_1+m_i}$.  If such a
  pseudo-reflection is a homology, then it has order dividing $N$, so is
  tame, while if it is a transvection, then its action on $L$-torsion has
  order $L$, so is again tame.  Thus the sheaf has tame ramification and
  the ramification divisors are disjoint, so form a normal crossings
  divisor, and thus the monodromy in characteristic $p$ is isomorphic to
  the monodromy in characteristic 0.
\end{proof}

By semicontinuity of monodromy, this immediately implies the analogous
statement for the fixed-curve family, as well as for any varying-curve
family that {\em has} a ramification point.  It thus remains only to
consider the case of (geometrically connected) unramified
$\mu_N$-pseudocovers.  Here we observe that the argument we used to extend
from genus $g-1$ to genus $g$ works equally well in good characteristic,
and thus we may reduce the unramified genus $g$ case to any ramified genus
$g-1$ case with two ramification points.

\begin{cor}
  For $g>3$, Pryms of geometrically connected unramified $\mu_N$-covers of
  smooth genus $g$ curves have large monodromy.
\end{cor}

\begin{rem}
  As before, this extends without too much difficulty to the
  case $g=3$.
\end{rem}

\subsection{Monodromy in bad characteristic}

The main issue with the varying-point family in bad characteristic is that
although there is an isomorphism between the spaces of ramification data at
different points of $C$, that isomorphism is not canonical (unless
$g(C)=1$), and thus there is no general way to define the family.  The
exception, of course, is when the moving point is tamely ramified, since
then the data is just an integer.  It turns out that those cases can still
be deduced from largness in characteristic 0, though the argument is quite
a bit more delicate.  The biggest issue is that the \'etale sheaf has wild
ramification, and (relatedly) the ramification divisor does not have normal
crossings.  We can kill the wild ramification by a suitable cyclic cover,
which for suitable choices of the tame part of the cover gives a new family
with normal crossings ramification.  Such choices of the tame part tend
{\em not} to produce covers with good reduction, but we can still arrange
to the get the correct {\em Prym}.

In this section, we suppose that $N$ factors as $p^lN'$ with $\gcd(p,N')=1$.

Note that since the point has tame ramification, we are only varying the
tame part of the cover, and thus the family of covers is the usual
$\Pic(C)[d]$-torsor over $C$.

\begin{lem}\label{lem:ramdeg}
  Let $R$ be a dvr, and let $\hat{C}_R$ be a smooth projective curve with
  geometrically irreducible fibers with an action of $\Z/N\Z$ and quotient
  $C_R$, such that $\Z/N\Z$ acts faithfully on $\hat{C}_k$.  Suppose
  moreover that the ramification points of $C_K$ are defined over $K$, and
  let $\bar{x}$ be a ramification point of $C_k$ of degree $m$.  Then
  the ramification points lying over $\bar{x}$ have degree dividing $m$,
  and at least one has degree $m$.
\end{lem}

\begin{proof}
  Denote the generator of $\Z/N\Z$ by $g$, so that the hypothesis on
  $\bar{x}$ is that it is the image of a point of $\hat{C}_k$ fixed by
  $g^{N/m}$ but not by any smaller power of $g$.  Let $C'$ be the quotient
  of $\hat{C}$ by $\langle g^{N/m}\rangle$.  Then $C'$ is a
  $\Z/m\Z$-cover of $C$, and we can split the desired conclusion into two
  pieces: (1) that $\hat{C}_K/C'_K$ is unramified at points reducing to
  preimages of $\bar{x}$ and (2) that $C'_K/C_K$ is totally ramified at
  some point reducing to $\bar{x}$.  In the first case, we may assume WLOG
  $C'=C$, so $m=1$, while in the second case we may assume WLOG that
  $\hat{C}=C'$, so $m=N$.

  For both claims, we note that for any nontrivial finite automorphism $h$
  of $\hat{C}_R$, the $h$-fixed subscheme of $\hat{C}_R$ is a Cartier
  divisor.  Indeed, we may write that subscheme as the intersection in
  $\hat{C}_R\times_R \hat{C}_R$ of the diagonal and the graph of $h$, both
  of which are Cartier divisors since $\hat{C}_R$ is smooth.  Since $h$
  acts nontrivially on $\hat{C}_R$ and $\hat{C}_R$ is irreducible, the
  intersection is itself a Cartier divisor.

  For the first claim, we note that since Cartier divisors are closed, for
  any nonidentity element $h\in \Z/N\Z$, and any fixed point of $h$ on
  $\tilde{C}_K$, its closure in $\hat{C}_R$ is also fixed by $h$, and thus
  the reduction to $\hat{C}_k$ is fixed.  In particular, if there is a
  ramified point with limit $\bar{x}$, then its preimages in $\hat{C}_K$
  are fixed by some $h$ and thus so is $\bar{x}$.

  For the second claim, since Cartier divisors have codimension 1, the
  preimage $\bar{y}$ of $\bar{x}$ cannot be an isolated fixed point of $g$
  in $\hat{C}_R$, but since $\Z/N\Z$ acts faithfully on $\hat{C}_k$, it
  {\em is} an isolated fixed point of $g$ in $\hat{C}_k$.  In particular,
  the generic point of any component of the $g$-fixed subscheme containing
  $\bar{y}$ is a ramified point over $K$ as required.
\end{proof}

We will need to know when a $\Z/p^l\Z\times \mu_{N'}$-cover has good
reduction.  Since this is hard for $\Z/p^l\Z$-covers, we restrict our
attention to understanding additional conditions that suffice given that
the intermediate $\Z/p^l\Z$-cover has good reduction.

\begin{lem}
  Let $R$ be a mixed characteristic dvr of residue characteristic $p$, and
  let $\hat{C}_R\to C_R$ be a smooth $\Z/p^l\Z$-cover such that the
  ramification points in $C_K$ are defined over $K$.  Further let
  $\hat{C}'_R\to C_R$ be a ramified $\mu_{N'}$-cover such that $\hat{C}'_k$
  is smooth and the ramification points in $C_K$ are defined over $K$.
  Suppose moreover that for every ramification point $x\in C_K$ of
  $\hat{C}'_K\to C_K$, the ramification degree of $\bar{x}$ for
  $\hat{C}_k\to C_k$ agrees with that of $x$ for $\hat{C}_K\to C_K$.  Then
  the normalization of $\hat{C}\times_C \hat{C}'$ has good reduction.
\end{lem}

\begin{proof}
  We may view the compositum as a $\mu_{N'}$-cover of $\hat{C}_K$, which
  has good reduction as long as the ramification points remain
  distinct on the special fiber.  Since $\hat{C}'_K$ has good reduction,
  the images of the ramification points in $C_K$ have distinct reductions,
  so it remains only to show that this continues to hold on $\hat{C}_K$, or
  equivalently that the number of preimages in $\hat{C}_K$ is the same as
  in $\hat{C}_k$.  But this follows immediately from the condition on
  ramification degrees.
\end{proof}

\begin{thm}
  Let $C/k$ be a smooth projective curve over a field of characteristic
  $p$, and $\hat{C}\to C$ a smooth $\Z/p^lZ\times \mu_{N'}$-cover of $C$,
  where $\gcd(N',p)=1$.  Suppose moreover that $\hat{C}\to C$ has at least
  one tame ramification point, and consider the family of covers obtained
  by varying that point.  If the total multiplicity of the ramification
  points of $C$ is at least $7-2g$, then the corresponding family of Pryms
  has large geometric monodromy.
\end{thm}

\begin{proof}
  Let $\hat{C}/R$ be the lift to characteristic 0.  Since we have been
  considering geometric monodromy, we may feel free to assume that $R$
  contains $\zeta_{p^l}$ so that we may view $\hat{C}_K$ as a
  $\mu_{p^lN'}$-cover.  Similarly, we may feel free to extend $R$ as
  necessary to ensure that the ramification points of $C_K$ are defined
  over $K$, so extend to sections $x_i$ of $C_R$.  Each such section meets
  $C_k$ in a ramification point, and the fact that $\hat{C}$ is smooth
  ensures that the intermediate $\mu_{N'}$-cover is smooth, and thus
  that the ramification points of the intermediate cover remain distinct
  in $C_k$.

  Now, suppose we modify the weights of $\hat{C}_K/C_K$ so that (a) the
  intermediate $\mu_{p^l}$-cover does not change and (b) for each point
  $\bar{x}$ of $C_k$, the sum of the weights of the ramification points
  reducing to $\bar{x}$ does not change.  Since $\bar{x}$ either has a
  single point over it, unramified for the $\mu_{p^l}$-cover, or has
  multiple ramification points over it, all ramified for the
  $\mu_{p^l}$-cover, we see that the modified cover $\hat{C}'_K$ has the
  same number of ramified points as $\hat{C}_K$ and thus has the same Prym
  dimension.  Of course, the special fiber of the normalization over $R$
  may become quite singular, but the {\em normalization} of the special
  fiber is the same as for $\hat{C}_K$, and thus it, too, has the same Prym
  dimension.  This remains true if we allow the tame points to move, and
  thus gives us a large collection of different varying-point families for
  which the Prym has the desired reduction mod $p$.
  
  We choose $\hat{C}'_K$ in the following way.  For each set of coalescing
  ramification points, choose one that maximizes the $p$-part of its
  ramification degree (which maximum by Lemma \ref{lem:ramdeg} agrees with
  the $p$-part of the ramification degree of the limit).  Then we can set
  the tame parts of the exponents (i.e., their reductions modulo $N'$) at
  the other points arbitrarily, and may thus in particular ensure that
  $m_1+m'_i=0(N')$ at each such point.  Let $\hat{C}'_K$ be the resulting
  curve.
  
  Now, let $X_R$ be the corresponding varying-point family, which we note
  is an \'etale cover of $C/R$.  The torsion-of-Prym sheaf on $X_K$ is
  ramified at the preimages of the fixed ramification points, with the
  ramification at $x_i$ having order $N/\gcd(m_1+m'_i,N)$.  Now, let
  $Y:=X\times_C \hat{C}_{p^l}$ where $\hat{C}_{p^l}$ is the
  intermediate $p^l$-cover.  The ramification degree of $Y/X$ at $x_i$ is
  thus given by
  \[
  p^l/\gcd(m_i,p^l) = p^l/\gcd(m'_i,p^l) = p^l/\gcd(m_1+m'_i,p^l).
  \]
  Thus if $y_i$ is a point over $x_i$ on $Y$, then the ramification of the
  torsion-of-Prym sheaf on $Y_K$ at $y_i$ has order
  \[
  (N/\gcd(m_1+m'_i,N))/(p^l/\gcd(m_1+m'_i,p^l))
  =
  N'/\gcd(m_1+m'_1,N'),
  \]
  so is prime to $p$.  Moreover, our choices of exponents ensure that this
  degree is 1 (so that the sheaf extends!) at all but at most one of the
  ramification points from each coalescing set.

  In other words, we have an \'etale sheaf on $Y_K$ with tame monodromy
  ramified along sections that remain distinct on $Y_k$, and thus the two
  monodromy groups agree.  Moreover, since $Y/X$ is Galois, the monodromy
  over $Y_K$ is the kernel of a morphism to the Galois group $\Z/p^l\Z$ and
  thus is still large, so that the same holds for $Y_k$.  Since we have
  shown that the extension of the sheaf to $X_k$ (and thus to $Y_k$) agrees
  with the Prym of the desired varying-point family, the claim follows.
\end{proof}

\begin{rem}
  The bound on the ramification multiplicity translates to the usual bound
  that the rank of the Prym as a $\Z[\zeta_N]$-modules be at least $5$.
\end{rem}

We can also obtain cases without tame ramification points (including the
unramified case) via the node-smoothing construction.  The main cost is
that the nodes being glued must be tamely ramified, and thus at least one
must be allowed to vary.  As a result, we only obtain results in the
varying curve scenario, and cannot deal with the case that $N$ is a power
of $p$.  In addition, if $N=p^lq^m$, then the monodromy of the relevant
genus $g-1$ family is reducible mod the primes over $q$, so we must exclude
that Tate module.

\begin{thm}
  Let $g>0$, $N'>1$.  For any $n\ge 0$ and any $n$-tuple of classes in
  $H^1_0(\A^1;\mu_{N'}\times \Z/p^l\Z)$, consider the family of global
  covers of smooth genus $g$ curves with $n$ ramification points having the
  specified analytic isomorphism classes.  Then the monodromy of the
  corresponding family of Pryms is large, as long as the total ramification
  multiplicity is at least $7-2g$ and we exclude the $q$-adic Tate module
  when $N'$ is a power of $q$.
\end{thm}

\begin{rem}
  Again, we could reduce the bound on the multiplicity to $6-2g$ without
  too much difficulty.
\end{rem}

\begin{rem}
  The exclusion of the $q$-adic Tate module is presumably just a
  technicality: the genus $g-1$ family has the same monodromy as in
  characteristic 0, and thus its $q$-adic monodromy contains (by
  considering the degenerate fiber in characteristic 0 where the two glued
  points merge) a subgroup with large image on a corank 2 sublattice.
  (Here assume the ramification multiplicity in genus $g$ is at least
  $9-2g$ so that the corank 2 sublattice will still have large monodromy.)
  Thus the only way the $q$-adic monodromy could fail to be large is by
  actually being reducible.
\end{rem}

\subsection{Loose ends}

We note that there are several gaps in our coverage of wildly ramified
covers.  The most significant is that since we reduce everything to cases
with at least one tame ramification point, we cannot deal with {\em purely}
wild covers; in particular, we do not obtain any results for unramified
$p^l$-covers in characteristic $p$.  In addition, we can say nothing about
general $p^lN'$-covers when the base curve is fixed (even in the $\P^1$
case!).  Finally, when $N'$ is a power of $q$, we have the presumably
technical issue that we could a priori fail to be large $q$-adically.

For these questions, the approach of \cite{AchterJD/PriesR:2007} may prove
fruitful.  Their method makes it quite straightforward to give inductive
proofs of {\em irreducibility} of monodromy, which in characteristic 0
tends to hold even when the monodromy isn't large, making the base cases in
principle more accessible.  Given irreducibility, largeness should follow
for any case that degenerates to a case of more than half the rank with
large monodromy.  (This is easier to find than the two compatible
degenerations with large monodromy required by
\cite{AchterJD/PriesR:2007}.)

\section{Selmer groups of elliptic surfaces with constant $j$}

Let $E/C/\F_q$ be an elliptic surface with constant $j$-invariant $j_0$,
where if $j_0=0$ we assume that the characteristic is not $2$ or $3$.  If
$E_{j_0}/\F_q$ is some fixed curve with that $j$-invariant, then there is
an induced cyclic cover of $C$, which at geometric points where $E$ has
good reduction is the $\Aut(E_{j_0,\bar\F_q})$-torsor of isomorphisms
between the given fiber and $E_{j_0}$.  This of course depends on the
choice of $E_{j_0}$, but we see that the {\em pseudocover} of $C$ is
independent of this choice, so is completely intrinsic to $E$.  Moreover
any choice of cover over the given pseudocover gives a curve $\hat{C}$
such that the minimal proper regular model of $E\times_C \hat{C}$ is
isomorphic to $E_{j_0}\times \hat{C}$ for some $E_{j_0}$; one can then
recover $E$ as the quotient of this product by the \'etale group
scheme $\Aut(E_{j_0})\cong \mu_N$, $N\in \{2,4,6\}$.

In other words, for each elliptic curve $E_{j_0}$ with cyclic automorphism
group, there is a natural identification between the moduli stack of
elliptic surfaces with $j$-invariant $j_0$ and the moduli stack of cyclic
covers.  For any family of such covers, we can use the corresponding family
of elliptic surfaces to construct additional representations of the
fundamental group of the family.  In particular, we may consider the {\em
  Selmer groups} of the surface, which under reasonable conditions
correspond to \'etale sheaves on the family.  We can also consider the
sheaf of sets obtained by quotienting by the cyclic group action, and find
that that sheaf descends to the corresponding family of pseudocovers.

It turns out that this sheaf of sets is canonically isomorphic to one we
have already considered: for $L$ prime to $N$, the sheaf of sets
corresponding to the $L$-Selmer group turns out to be the same as the sheaf
of sets corresponding to the $L$-torsion of the Prym.  In other words, the
two monodromy representations are projectively isomorphic.  We have already
seen that projective largeness implies largeness, and thus this will let us
see that the monodromy of the $L$-Selmer group is large (apart from certain
very small cases).

The key idea is that there are maps between the Selmer groups of an abelian
variety and the Selmer groups of its base changes, which in many cases let
us write the former in terms of the latter.  So in many cases, the
computation of the Selmer group of an isotrivial family can be reduced to
the analogous computation for a {\em trivial} family.

The first step is to observe that if we allow abelian varieties, then the
Selmer group over a (finite separable) field extension can be written as a
Selmer group over $K$.

\begin{lem}
  Let $L/K$ be a finite separable extension of global fields, and let
  $\phi:A\to B$ be an isogeny of abelian varieties over $L$.  Let
  $\Res_{L/K}\phi$ denote the corresponding isogeny of Weil restrictions to
  $K$.  Then there is a canonical isomorphism
  \[
  \Sel(\phi)\cong \Sel(\Res_{L/K}\phi).
  \]
\end{lem}

\begin{proof}
  We first note that over any field $K'$ extending $K$, we have canonical
  isomorphisms
  \begin{align}
  H^1_{\fl}(K';\ker\Res_{L/K}\phi)
  &\cong
  H^1_{\fl}(K';\ker\Res_{L\otimes_K K'/K'}\phi)\notag\\
  &\cong
  H^1_{\fl}(K';\Res_{L\otimes_K K'/K'}\ker\phi)\notag\\
  &\cong
  H^1_{\fl}(L\otimes_K K';\ker\phi).
  \end{align}
  In particular, for any place $v$ of $K$, we have a commutative diagram
  \[
  \begin{CD}
  H^1_{\fl}(K;\ker\Res_{L/K}\phi)
  @>>>
  H^1_{\fl}(K_v;\ker\Res_{L/K}\phi)\\
  @VVV @VVV\\
  H^1_{\fl}(L;\ker\phi)
  @>>>
  \prod_{w/v} H^1_{\fl}(L_w;\ker\phi)
  \end{CD}
  \]
  in which the vertical maps are the canonical isomorphisms.  Since
  $\Sel(\Res_{L/K}\phi)$ is by definition the kernel of the morphism
  \[
  H^1_{\fl}(K;\ker\Res_{L/K}\phi)
  \to
  \prod_v H^1_{\fl}(K_v;\ker\Res_{L/K}\phi),
  \]
  the isomorphisms identify it with the kernel of the morphism
  \[
  H^1_{\fl}(L;\ker\phi)
  \to
  \prod_v \prod_{w/v} H^1_{\fl}(L_w;\ker\phi)
  =
  \prod_w H^1_{\fl}(L_w;\ker\phi),
  \]
  which is $\Sel(\phi)$ as required.
\end{proof}

For $A$ defined over $K$, we have a canonical diagonal morphism $A\to
\Res_{L/K}A$ as well as a canonical trace morphism $\Res_{L/K}A\to A$, such
that the composition is multiplication by $\deg(L/K)$.  By functoriality,
this induces morphisms
\[
\Sel(\phi)\to \Sel(\phi_L)\to \Sel(\phi),
\]
composing to multiplication by $\deg(L/K)$.  (Note that the analogous maps
on Galois cohomology are nothing other than restriction and transfer.)  If
$L/K$ is Galois, there is an induced action of $\Gal(L/K)$ on
$\Res_{L/K}A$, and the maps $A\to \Res_{L/K}A\to A$ are equivariant for
this action and the trivial action on $A$.  This implies that we have
morphisms
\[
H_0(\Gal(L/K);\Sel(\phi_L))\to \Sel(\phi)\to H^0(\Gal(L/K);\Sel(\phi_L))
\]
where the composition is the canonical map
\[
H_0(\Gal(L/K);\Sel(\phi_L))\to H^0(\Gal(L/K);\Sel(\phi_L))
\]
induced by symmetrization, giving us a triangle of maps
\[
H_0(\Gal(L/K);\Sel(\phi_L))\to \Sel(\phi)\to H^0(\Gal(L/K);\Sel(\phi_L))
\to H_0(\Gal(L/K);\Sel(\phi_L))\to\cdots
\]
such that the composition once around the triangle starting from any point
is multiplication by $\deg(L/K)$.

In the application of interest, $\phi_L$ is an isogeny of constant abelian
varieties.

\begin{prop}
  Let $\phi:A\to B$ be an isogeny of abelian varieties defined over a
  separably closed field $k$, let $C/k$ be a curve, and suppose $\deg\phi$
  is invertible in $k$.  Then there is a canonical isomorphism
  \[
  \Sel_\phi(A_{k(C)}) \cong J(C)[\deg\phi]\otimes_{\Z}
  \mu_{\deg\phi}^{-1}\otimes_{\Z} \ker\phi.
  \]
\end{prop}

\begin{proof}
  Since $A_{k(C)}$ is constant, it has good reduction everywhere, and thus
  the local condition at any place $v$ of $k(C)$ for a cohomology class to
  lie in the Selmer group is nothing other than the requirement that it be
  unramified.  In other words, we have
  \[
  \Sel_\phi(A_{k(C)})\cong H^1_{\fl}(C;\ker\phi).
  \]
  Since $\ker\phi$ is a constant \'etale sheaf annihilated by $L=\deg\phi$,
  we have
  \[
  H^1_{\fl}(C;\ker\phi)
  \cong
  H^1_{\fl}(C;\mu_L)\otimes \mu_L^{-1}\otimes \ker\phi
  \cong
  \Pic(C)[L]\otimes \mu_L^{-1}\otimes \ker\phi.
  \]
\end{proof}

If $\phi$ and $C$ are defined over a non-closed field $k$, then the
Proposition gives an isomorphism of $\Gal(k)$-modules.  In nice cases, we
can recover the Selmer group over $k$ from the corresponding
$\Gal(k)$-module.

\begin{prop}
  Let $\phi:A\to B$ be an isogeny of abelian varieties defined over a
  global function field $K$, and let $\tilde\phi:{\cal A}\to {\cal B}$
  denote the corresponding morphism of N\'eron models.  If $\phi$ induces
  an isomorphism on component groups, then $\Sel_\phi(A)$ may be identified
  with the first hypercohomology of $\tilde\phi$, viewed as a complex of
  \'etale sheaves.
\end{prop}

\begin{proof}
  The cohomology group $H^1_{\fl}(K;\ker\phi)$ agrees with the first
  hypercohomology of $\phi$, which classifies $A$-equivariant morphisms
  $X\to B$ where $X$ is an $A$-torsor and the action of $A$ on $B$ is
  induced by $\phi$.  (I.e., it classifies $A$-torsors together with an
  explicit trivialization of the induced $B$-torsor.)  The Selmer group
  then classifies morphisms such that $X$ is trivial at every completion.
  This in particular implies that $X$ is unramified, and thus extends to an
  ${\cal A}$-torsor, which by the universality of N\'eron models gives us a
  morphism ${\cal X}\to {\cal B}$, i.e., a class in the hypercohomology
  group $H^1_\et(C;\tilde\phi)$.  Such a class can fail to be in the Selmer
  group: $X$ is unramified at the place $v$ iff $X(K_v^{\text{unr}})$ is
  nonempty, while to be in the Selmer group requires that $X(K_v)$ itself
  be nonempty.  The obstruction at a given place $v$ is nothing other than
  \[
  H^1_{\et}(k_v;{\cal A}) \cong H^1_{\et}(k_v;\pi_0({\cal A}))
  \]
  and thus we see in general that the Selmer group is the kernel of the
  morphism
  \[
  H^1_\et(C;\tilde\phi)\to H^1_\et(C;\pi_0({\cal A})),
  \]
  noting that the sheaf $\pi_0({\cal A})$ is supported at finitely many
  points.  The morphism of complexes comes from the chain map
  \[
  \begin{CD}
    {\cal A}@>\tilde\phi >> {\cal B}\\
    @VVV @VVV\\
    \pi_0{\cal A}@>0>> 0
  \end{CD}
  \]
  which factors through the chain map
  \[
  \begin{CD}
    {\cal A}@>\tilde\phi >> {\cal B}\\
    @VVV @VVV\\
    \pi_0({\cal A})@>\pi_0(\tilde\phi) >> \pi_0({\cal B}).
  \end{CD}
  \]
  In particular, we see that if $\pi_0(\tilde\phi)$ is an isomorphism, then
  the latter chain map is 0, forcing
  \[
  H^1_\et(C;\tilde\phi)\to H^1_\et(C;\pi_0({\cal A}))
  \]
  to vanish.
\end{proof}

\begin{rem}
In the case of isotrivial families of elliptic curves, we find that the
condition for $\pi_0(\tilde\phi)$ to be an isomorphism is precisely the
same as the condition for the monodromy of the Prym to be irreducible.  For
instance, for $j=0$, the component group at a ramification point $x$ of
weight $1,2,3,4,5$ is $1,C_3,C_2^2,C_3,1$ (corresponding to Kodaira types
$\text{II},\text{IV},\text{I}_0^*,\text{IV}^*,\text{II}^*$).
\end{rem}

\begin{prop}
  Let $\phi:A\to B$ be an isogeny of abelian varieties defined over a
  global function field $K=k(C)$, and let $\tilde\phi:{\cal A}\to {\cal B}$
  denote the corresponding morphism of N\'eron models, which we assume
  induces an isomorphism on component groups.  If $\ker\phi(k^\sep(C))=0$,
  then
  \[
  \Sel_\phi(C)\cong \Sel_\phi(C_{k^\sep})^{\Gal(k)}.
  \]
\end{prop}

\begin{proof}
  The Lyndon-Hochschild-Serre spectral sequence gives us a five-term exact
  sequence relating the hypercohomology of $\tilde\phi$ over $k(C)$ and the
  hypercohomology over $k^\sep(C)$:
  \[
  0\to H^1(k;H^0_{\et}(C_{k^\sep};\tilde\phi))
   \to H^1_\et(C_k;\tilde\phi)
   \to H^0(k;H^1_{\et}(C_{k^\sep};\tilde\phi))
   \to H^2(k;H^0_{\et}(C_{k^\sep};\tilde\phi)).
  \]
  We thus see that the desired isomorphism holds as long as
  \[
  H^0_{\et}(C_{k^\sep};\tilde\phi)=0,
  \]
  which is equivalent (by the extension property of N\'eron models) to the
  vanishing of $\ker\phi(k^\sep(C))$.
\end{proof}

\begin{rem}
  Note that the additional hypothesis typically failes for trivial
  families.  For isotrivial families, the condition often holds, however,
  in particular when $\deg\phi$ is prime to the degree of the Galois
  extension where the family becomes trivial.  Even that is often
  unnecessary: e.g., for the $j=0$ family, the condition at 2 is unneeded
  if the intermediate $3$-cover is nontrivial, and the condition at $3$ is
  unneeded if the intermediate $2$-cover is nontrivial.
\end{rem}

Returning to our elliptic curve example, let $A=E_0$ be some fixed curve
with $j=0$, say $y^2=x^3+1$, and suppose that $L$ is prime to $6$. Then
elliptic surfaces with $j$-invariant 0 over $C$ (or, when $p|6$, a certain
subfamily of such surfaces) are in natural correspondence (given our choice
of $E_0$) with $\mu_6$-covers $\hat{C}$ of $C$, and we find that for such a
surface $E$, we have geometric Selmer group (i.e., Selmer group after base
changing to $\bar\F_q$)
\[
\Sel_L(E) = (J(\hat{C})[L]\otimes_{\Z} E_0[L]\otimes_{\Z} \mu_L^{-1})^{\mu_6}
\]
Since $\mu_6$ acts trivially on $\mu_L^{-1}$, we may pull that sheaf
outside (all it does is twist the Galois action).  As a $\Z/L\Z$-module,
$E_0[L]$ is free of rank 2, and is faithful as a $\mu_6$-module, and thus
is nothing other than
\[
\Z[\zeta_6]/L\Z[\zeta_6].
\]
In other words, if we ignore the Galois action, we have
\[
(J(\hat{C})[L]\otimes_{\Z} E_0[L]\otimes_{\Z} \mu_L^{-1})^{\mu_6}
\cong
J(\hat{C})[L]\otimes_{\Z[t]/(t^6-1)} \Z[t]/(t^2-t+1)
\cong
\Prym(\hat{C},\mu_6)[L].
\]
Since the isomorphism
\[
E_0[L]\cong \Z[\zeta_6]/L\Z[\zeta_6]
\]
is valid over some finite extension of $\F_q$, we see that the Galois
action is relevant to the {\em arithmetic} monodromy, but not the {\em
  geometric} monodromy.

In other words, given a {\em family} of $\mu_6$-covers $\hat{C}/C$ and
the associated family of $j=0$ elliptic surfaces $E$ over $C$, we have an
isomorphism of \'etale sheaves
\[
\Sel_L(E_{\bar\F_q(C)})\cong \Prym(\hat{C}_{\bar\F_q},\mu_6)[L]
\]
which is compatible with the Galois action over some finite extension
$\F_{q^m}/\F_q$.  In particular, we find that the geometric monodromy
groups are isomorphic.

\begin{thm}
  Fix $n\ge \max(5,2g-1)$, and let $\hat{C}\to C$ be a geometrically
  connected $\mu_6$-cover of a smooth genus $g$ curve of characteristic
  $p\nmid 6$ with $n+2-2g$ ramification points. Let $x_1$ be one of those
  points.  Then the monodromy of the family of Selmer groups of $j=0$
  curves obtained by varying $x_1$ is large away from $2$ and $3$.
\end{thm}

\begin{rems}
  We of course get analogous results for \'etale covers for $g\ge 3$, as
  long as we allow the base curve to vary.
\end{rems}

A similar argument gives us the following.

\begin{thm}
  Fix $n\ge \max(5,2g-1)$, and let $\hat{C}\to C$ be a geometrically
  connected $\mu_4$-cover of a smooth genus $g$ curve of characteristic
  $p\nmid 4$ with $n+2-2g$ ramification points. Let $x_1$ be one of those
  points.  Then the monodromy of the family of Selmer groups of $j=1728$
  curves obtained by varying $x_1$ is large.
\end{thm}

One can also consider cubic twists of $y^2=x^3+1$, for which one gets the
following.

\begin{thm}
  Fix $n\ge \max(5,2g-1)$, and let $\hat{C}\to C$ be a geometrically
  connected $\mu_3$-cover of a smooth genus $g$ curve of characteristic
  $p\nmid 3$ with $n+2-2g$ ramification points. Let $x_1$ be one of those
  points.  Then the monodromy of the family of Selmer groups of $j=0$
  curves obtained by varying $x_1$ is large.
\end{thm}

One can say something similar about surfaces with constant
$j\notin\{0,1728\}$.  Here it is natural to allow the $j$ invariant to vary
as well, in which case one finds that the monodromy is large in
$\Sp_n\otimes \Sp_2$.

We can strengthen the $j=0$ result in the most natural case, in which the
ramification points of $\hat{C}\to C$ all have weight 1 (or, for that
matter, 5).  In that case, for {\em any} $L$, the $L$-Selmer group is the
$L$-torsion portion of an $L^\infty$-Selmer group, and thus the monodromy
not only makes sense, but lies in a suitable ad\`elic group.

\begin{thm}
  Fix $n\ge \max(5,2g-1)$, and let $\hat{C}\to C$ be a geometrically
  connected $\mu_6$-cover of a smooth genus $g$ curve of characteristic
  $p\nmid 6$ with $n+2-2g$ ramification points, all of weight 1 or 5. Let
  $x_1$ be one of those points.  Then the monodromy of the family of Selmer
  groups of $j=0$ curves obtained by varying $x_1$ is large.
\end{thm}

\begin{proof}
  We still have restriction and transfer maps at $2$ and $3$, but now the
  maps are no longer isomorphisms: instead, their compositions are
  symmetrization or multiplication by 6 as appropriate.  However, this is
  enough to establish that the corresponding $2$- and $3$-adic lattices are
  at least {\em comparable}.  In other words, the two monodromy
  representations are isomorphic over $\Q_2$ and $\Q_3$.  But the monodromy
  of the Prym is irreducible modulo $2$ and $\sqrt{-3}$, which implies that
  the two lattices are at least {\em homothetic}, implying that
  the representations are actually isomorphic!
\end{proof}

Note that just as for the Prym, there are residual questions: for $g>0$, we
do not in general know the image of the determinant map.  Furthermore, the
{\em arithmetic} monodromy lies in some coset of the geometric monodromy in
its normalizer, but that coset will depend on the Galois module structure
of $E_0[L]\otimes \mu_L^{-1}$.  (Presumably this will become simpler if we
enlarge our family to allow $E_0$ to vary.)

The failure to understand the determinant should be a relatively minor
issue: experimentally, the different cosets of $\SU$ inside a given coset
of $\U$ give either the same distribution for the shape of the fixed
submodule or very nearly the same, all contained in an $L^1$ ball
with radius converging rapidly to $0$ as the rank goes to infinity.  If
this is indeed the case, then the error introduced by replacing the average
over a coset of geometric monodromy by an average over the corresponding
coset of the unitary group will be small.  Moreover, we can easily tell
which coset of the unitary group contains the arithmetic monodromy: in the
Selmer case, the arithmetic monodromy preserves a quadratic form, and thus
is either contained in the unitary group or in the coset of semilinear
transformations negating the anti-Hermitian form, depending on how
$\Gal(\F_q)$ acts on $\mu_6$.

We have not tried to compute the actual limiting distributions, however, as
this still involves four cases (or six if we include the $3$-power Selmer
group for $j=0$): the behavior for $l$-powers depends on whether $l$ splits
in $\Q(\zeta_6)$ as well as whether $q$ is $1$ or $2$ mod 3.  Note that
when $q=1(3)$, it is straightforward to use the above monodromy
calculations to estimated the expected size of the $l$-Selmer group.  The
key is the following (presumably well-known) analogue of Burnside's Lemma.

\begin{lem}
  Let $G$ be a finite group acting on the finite set $X$, let $N\normal
  G$ be a normal subgroup, and let $C$ be a coset of $N$ in $G$.  Then
  \[
  \frac{1}{|C|} \sum_{g\in C} |\{x:x \in X|gx=x\}|
  \]
  is equal to the number of $N$-orbits in $X$ that are preserved by $C$.
\end{lem}

\begin{proof}
  We may write $|\{x:x\in X|gx=x\}|$ as $\Tr(\pi(g))$, where $\pi$ is the
  corresponding permutation representation of $G$, and thus we reduce to
  computing
  \[
  \Tr(\frac{1}{|C|} \sum_{g\in C}\pi(g)).
  \]
  If we define
  \[
  \Pi_N:=\frac{1}{|N|} \sum_{g\in N}\pi(n),
  \]
  so that $\Pi_N$ projects onto the $N$-fixed subspace, we find
  \[
  \frac{1}{|C|} \sum_{g\in C}\pi(g)
  =
  \pi_N \bigl(\frac{1}{|C|} \sum_{g\in C}\pi(g)\bigr) \pi_N,
  \]
  and thus we may instead compute the trace of the restriction to the
  $N$-fixed subspace, which has a basis given by the sums over orbits of
  $N$ in $X$.  In that basis, $g\in C$ acts by a permutation, and thus the
  trace of $g$ on that subspace is nothing other than the number of fixed
  points of that permutation.
\end{proof}

\begin{lem}
  Let $G/\F_q$ be a smooth connected group scheme and let $X/\F_q$ be a
  $G$-set such that every point in $X(\F_q)$ has smooth connected
  stabilizer.  Then the set of $G(\F_q)$-orbits in $X(\F_q)$ is naturally
  bijective with the set of $G(\bar\F_q)$-orbits in $X(\bar\F_q)$ that are
  preserved by the action of Frobenius.
\end{lem}

\begin{proof}
  Let $x\in X(\bar\F_q)$ be a representative of an orbit preserved by
  Galois, so that $\sigma_q(x)=gx$ for some $g\in G(\bar\F_q)$.  This
  determines a Galois 1-cocycle for $H^1(\F_q;G)$, which is trivial since
  $G$ is connected \cite{LangS:1956}.  There is thus some $h\in
  G(\bar\F_q)$ such that $g = \sigma_q(h)^{-1} h$ and thus
  $\sigma_q(hx)=hx$, so that there is an orbit representative defined over
  $\F_q$.

  It remains to show that that orbit representative is unique up to the
  action of $G(\F_q)$; in other words, if $\sigma_q(x)=x$ and
  $\sigma_q(gx)=gx$, we need to show that $gx=g'x$ for some
  $g'\in G(\F_q)$.  By assumption, $g^{-1}\sigma_q(g)\in G_x$, and
  connectivity of $G_x$ again implies that $g^{-1}\sigma_q(g) =
  h \sigma_q(h)^{-1}$ for some $h\in G_x$, and thus we may take $g'=gh$.
\end{proof}

\begin{lem}
  For $n\ge 3$, the orbits of $\SL_n(\bar\F_q)$ on $V_n\oplus V_n^\perp$
  are distinguished by the evaluation map $(v,\lambda)\mapsto \lambda(v)$,
  except that if that map is 0, there are four orbits depending on which of
  $v$, $\lambda$ or both are 0.  In each case, the orbit stabilizer is
  smooth and connected, and every orbit is preserved by $\GL_n(\bar\F_q)$.
\end{lem}

\begin{proof}
  Certainly, $\SL_n$ is transitive on nonzero vectors, so that every orbit
  with $v\ne 0$ has a representative with $v=e_1$ (relative to ones
  favorite basis).  Since $\SL_n$ is transitive on nonzero linear
  functionals, there are precisely two orbits with $v=0$, as required.

  If the expansion of $\lambda$ relative to the dual basis is $c f_1 +
  \sum_{i>1} a_i f_i$ with $c\ne 0$, then $e_i \mapsto e_i + (a_i/c) e_1$
  eliminates $a_i$, and thus we have an orbit representative $(e_1,cf_1)$,
  proving that there is indeed a unique orbit for each nonzero value of
  $\lambda(v)$.   Similarly, if $c\ne 0$, then either $\lambda=0$ or we may
  use transitivity of $\SL_{n-1}$ on nonzero linear functionals to reduce
  to the orbit representative $(e_1,f_n)$.

  Since $\GL_n(\bar\F_q)$ is generated over $\SL_n(\bar\F_q)$ by scalars,
  which act in opposite ways on the two summands, $\lambda(v)$ is a
  $\GL_n$-invariant, and thus the orbits are indeed all preserved by
  $\GL_n(\bar\F_q)$.  For the orbits with $\lambda(v)\ne 0$, the stabilizer
  of the standard representative is nothing other than $\SL_{n-1}$, while
  for the orbits with $\lambda(v)=0$, the stabilizer is parabolic, and in
  either case, the stabilizer is a connected group scheme as required.
\end{proof}

\begin{thm}\label{thm:avgsel_j0}
  Fix an integer $n>1$, and let $q_m$ be a sequence of odd prime powers
  congruent to $1$ mod $3$ with $\lim_{m\to\infty} q_m=\infty$.  For each
  $m$, let $X_m/\P^1$ range over all elliptic surfaces of height $n$ (i.e.,
  with discriminant in $O_{\P^1}(12n)$) with $j$-invariant $0$, such that
  the N\'eron model has no reducible fibers.  Then for all primes $l\ne 3$, as
  $m\to\infty$,
  \[
  E|\Sel_l(X_m)| = l+2+\left(\frac{l}{3}\right) + O(q_m^{-1/2})
  \]
\end{thm}

\begin{proof}
  By equidistribution, this reduces to showing that the average of
  $|\ker(g-1)|$ over the arithmetic monodromy is
  $l+2+\left(\frac{l}{3}\right)$.  Since the arithmetic monodromy contains
  the appropriate unitary (or linear) group, it suffices to show that this
  holds for any coset of the special unitary group, or equivalently that
  the special unitary group has that many orbits, all of which are
  preserved by the general unitary group.  The special unitary group is the
  $\F_l$-points of a group scheme geometrically isomorphic to $\SL_n$, and
  thus we may apply the previous Lemmas to reduce to counting the number of
  orbits over $\bar\F_l$ that are preserved by Frobenius.  The orbits for
  which both $v$ and $\lambda$ are nonzero are classified by $\lambda(v)$,
  on which Frobenius acts as expected, so that the $\F_l$-rational orbits of
  that form are bijective with $\F_l$.  Similarly, the singleton orbit
  $(v,\lambda)=(0,0)$ is preserved by Frobenius.  Whether the remaining two
  orbits contribute is determined by whether Frobenius swaps $V_n$ and
  $V^\perp_n$, which happens precisely when $l=2(3)$.  We thus see that
  there are $l+1$ orbits when $l=2(3)$ and $l+1+2$ orbits when $l=1(3)$.
\end{proof}

\begin{rem}
  One can show that the formula continues to hold for $l=3$.
\end{rem}

We thus find that one already sees a significant difference in distribution
for half the primes $l$.

\medskip

It should also be possible to use these results to show that various
families of general curves have large Selmer monodromy.  Indeed, for any
family of general curves such that the subfamily with $j=0$ has large
monodromy, the monodromy of the larger family contains the monodromy of the
subfamily, and thus contains a conjugate of the relevant special unitary
group $\SU_n$.  For each $l$, this is close to a maximal subgroup of
$\SO_{2n}(\F_l)$, and thus nearly any way of showing the monodromy is {\em
  not} contained in $\U_n(\F_l)$ will imply largeness.  (Largeness mod each
prime $l$ suffices, for much the same reason as it does for Pryms.)  For
instance, if the monodromy of the general family contains a reflection
(e.g., if it contains curves with precisely one reducible fiber, of type
$\text{II}$ or $\text{I}_2$), this immediately implies largeness of
monodromy.  Indeed, the proof of Theorem \ref{thm:avgsel_j0} shows that the
special unitary group is transitive on anisotropic vectors of any given
norm, and thus if the monodromy contains any reflection, it contains an
entire conjugacy class of reflections.

\bibliographystyle{plain}

\end{document}